\documentclass[11pt]{article}
\usepackage[a4paper,centering,hscale=0.7,vscale=0.75]{geometry}
\usepackage{amsmath,amsthm,amsfonts,amssymb}

\newtheorem{prop}{Proposition}[section]
\newtheorem{thm}[prop]{Theorem}
\newtheorem{coroll}[prop]{Corollary}
\newtheorem{lemma}[prop]{Lemma}

\theoremstyle{remark}
\newtheorem{rmk}[prop]{Remark}

\newcommand{\ep}[1]{{#1}^\varepsilon}
\newcommand{\ip}[2]{\left\langle #1,#2 \right\rangle}
\newcommand{\norm}[1]{\left\| #1 \right\|}
\newcommand{\bnorm}[1]{\bigl\| #1 \bigr\|}

\newcommand{\lip}{\dot{C}^{0,1}}
\newcommand{\m}{\bar{\mu}}
\newcommand{\E}{\mathbb{E}}
\renewcommand{\H}{\mathbb{H}}

\renewcommand{\L}{\mathbb{L}}
\newcommand{\LL}{\mathcal{L}}
\renewcommand{\P}{\mathbb{P}}
\newcommand{\erre}{\mathbb{R}}
\newcommand{\enne}{\mathbb{N}}

\title{Approximation and convergence of solutions to semilinear
  stochastic evolution equations with jumps}

\author{Carlo Marinelli\thanks{Institut f\"ur Angewandte Mathematik,
    Universit\"at Bonn, Germany, and Facolt\`a di Economia, Libera
    Universit\`a di Bolzano, Italy. URL:
    \texttt{http://www.uni-bonn.de/$\sim$cm788}.} \and Luca Di
  Persio\thanks{Dipartimento di Matematica, Universit\`a di Trento,
    Italy.} \and Giacomo Ziglio\thanks{Dipartimento di Matematica,
    Universit\`a di Trento, Italy.}}

\date{\normalsize 24 May 2012}

\begin{document}
\maketitle

\begin{abstract}
  We prove that the mild solution to a semilinear stochastic evolution
  equation on a Hilbert space, driven by either a square
  integrable martingale or a Poisson random measure, is (jointly)
  continuous, in a suitable topology, with respect to the initial
  datum and all coefficients. In particular, if the leading linear
  operators are maximal (quasi-)monotone and converge in the strong
  resolvent sense, the drift and diffusion coefficients are uniformly
  Lipschitz continuous and converge pointwise, and the initial
  data converge, then the solutions converge.
\end{abstract}


\section{Introduction}
Consider the stochastic evolution equation
\begin{equation}
  \label{eq:0}
  du(t) + Au(t)\,dt + f(u(t))\,dt = B(u(t-))\,dM(t),
  \qquad u(0)=u_0,
\end{equation}
on a real separable Hilbert space $H$, where $A:D(A) \subset H \to H$
is a linear quasi-maximal monotone operator, $M$ is a
Hilbert-space-valued square integrable martingale, and the
coefficients $f$, $B$ satisfy suitable Lipschitz and linear growth
conditions (see below for precise assumptions on all data of the
problem). The purpose of this work is to provide sufficient conditions
for the (sequential) continuity, in an appropriate topology, of the
map $(u_0,A,f,G) \mapsto u$, where $u$ denotes the mild solution to
\eqref{eq:0}. The same problem is considered also for equations (still
with multiplicative noise) driven by compensated Poisson random
measures. Our main results are Theorems \ref{thm:nyo2} and
\ref{thm:nyop} below. The problem we consider, apart of having its own
intrinsic interest, is also motivated by several other considerations,
such as the study of the stability of models based on stochastic
partial differential equations (SPDEs) and the convergence of
numerical approximation schemes. Moreover, as is well known, a
technique to obtain estimates for mild solutions to SPDEs consists in,
first, approximating the unbounded operator $A$ by a bounded one (such
as e.g. the Yosida approximation), so that, roughly speaking, tools
from stochastic calculus for semimartingales can be applied to the
regularized equation; then, showing that the estimates ``pass'' to the
original equation. Such regularization procedure is needed because
mild solutions, in general, are not semimartingales, so that tools
like It\^o's formula are not directly applicable. Motivated mostly by
these considerations, continuous dependence on $A$ for stochastic
convolutions against Hilbert-space-valued Wiener processes was
established already in \cite{DPKZ} (cf. also \cite[Thm.~5.12]{DZ92}),
where the authors introduced the by now classical factorization
method. Several refinements of this result, all relying on the
factorization method, have appeared in the literature, the most
sophisticated of which is given in the recent work \cite{KvN1}, where
stochastic equations on UMD Banach spaces driven by a cylindrical
Wiener process are considered. One should also mention the related
results due to Gy\"ongy (see e.g. \cite{Gyo:aprx1,Gyo:stab1}) for
SPDEs in the variational setting driven by finite-dimensional continuous
martingales.

If the martingale $M$ is discontinuous, we are not aware of any
results about continuous dependence of the solution on the data (apart
of \cite{cm:JFA10}, where continuity and (Fr\'echet) differentiability
of the map $u_0 \mapsto u$ is investigated for equations with Poisson
noise). In fact, in the case of jump noise, the factorization method
is unfortunately no longer applicable, hence a different approach is
needed. The present paper provides such an alternative method, which
is in part inspired, perhaps somewhat unexpectedly, by techniques from
the theory of \emph{nonlinear} maximal monotone operators on Hilbert
spaces (in particular by Br\'ezis' proof in \cite{Bmax} of a nonlinear
version of Trotter-Kato's theorem). Our method, however, is restricted
to operators $A$ that are quasi-monotone, while the factorization
method is not. Therefore, in the context of equations driven by a
Wiener process, our method is not a replacement of the ``usual'' one.
Let us also mention that one can find in the literature very
satisfactory results on continuous dependence on the coefficients for
finite-dimensional stochastic differential equations driven by general
semimartingales, see e.g. \cite{Eme:stab} and
\cite[pp.~257-ff.]{Protter}. On the other hand, our continuity results
apply to those classes of stochastic evolution equations for which a
``decent'' well-posedness theory (in the mild sense) is available. In
other words, the gap with respect to the finite-dimensional results is
mainly due to the less developed well-posedness theory in infinite
dimensions.

Before concluding this introductory section with some words about
notation, let us give a brief overview of the paper: in Section
\ref{sec:main} we state the main results, whose proofs can be found in
Section \ref{sec:prove}. Section \ref{sec:prel} collects some facts
about (linear) maximal monotone operators on Hilbert spaces and on
stochastic integrals (and convolutions) with respect to
Hilbert-space-valued square integrable martingales. The core of the
paper are Sections \ref{sec:conv1} and \ref{sec:conv2}, where
continuity of stochastic convolutions with respect to the operator $A$
is established.  In particular, first we approximate $A$ by its Yosida
regularization and we prove that the correspoding stochastic
convolutions converge. Then we show that the same holds if instead of
the Yosida approximation we consider a sequence of maximal
quasi-monotone operators converging to $A$ in the strong resolvent
sense. In the last section we briefly comment on the case of equations
with additive noise.

\bigskip

\noindent
\textbf{Notation.} Given two normed spaces $E$, $F$, we
shall denote by $\lip(E,F)$ the space of Lipschitz continuous
functions from $E$ to $F$, i.e. the space of functions $\phi:E \to F$
such that
\[
\|\phi\|_{\lip(E,F)} := \sup_{x \neq y}
\frac{\|\phi(x)-\phi(y)\|_F}{\|x-y\|_E} < \infty.
\]
Whenever we write $\phi \in \lip(E,F)$, it is implicitly assumed, to
avoid nonsensical situations, that there exists $a \in E$ such that
$\|\phi(a)\|_F<\infty$. This immediately implies $\|\phi(x)\|_F \leq
N(1+\|x\|_E)$, with $N$ depending only on $\|\phi\|_{\lip(E,F)}$,
$\|a\|_E$, $\|\phi(a)\|_F$.  If $E$ and $F$ are complete, the space of
linear continuous operators, of trace class, and of Hilbert-Schmidt
operators from $E$ to $F$ will be denoted by $\LL(E,F)$, $\LL_1(E,F)$,
and $\LL_2(E,F)$, respectively.  If $E=F$, we shall simply write
$\LL(E)$ in place of $\LL(E,E)$, and similarly for other
spaces. Occasionally we shall drop the indication of the spaces $E$
and $F$ altogether if there is no risk of confusion. We shall write $a
\lesssim b$ if there exists a constant $N>0$ such that $a \leq Nb$. If
the constant $N$ depends on parameters $p_1,\ldots,p_n$, we shall also
write $N=N(p_1,\ldots,p_n)$ and $a \lesssim_{p_1,\ldots,p_n} b$.

\section{Main results}
\label{sec:main}
Let $H$ and $K$ two real separable Hilbert spaces. The inner product
and norm of $H$ will be denoted by $\ip{\cdot}{\cdot}$ and
$\|\cdot\|$, respectively. Let $A:D(A) \subset H \to H$ be a
linear (unbounded) maximal quasi-monotone operator, i.e. such that
\[
\ip{Ax}{x} + \eta \|x\|^2 \geq 0 \qquad\forall x\in D(A),
\]
for some $\eta>0$, and $R(\lambda I+A)=H$ for all $\lambda>\eta$
(range and domain of operators will be denoted by $R(\cdot)$ and
$D(\cdot)$, respectively). The strongly continuous semigroup of
quasi-contractions on $H$ generated by $-A$ will be denoted by $S$.

Let $T>0$ be fixed. All random variables and processes are assumed to
be defined on a filtered probability space
$(\Omega,\mathcal{F},\mathbb{F},\P)$,
$\mathbb{F}=(\mathcal{F}_t)_{0\leq t \leq T}$, satisfying the
``usual'' conditions. Statement involving random elements are always
meant to hold $\P$-a.s.. The space $L_p(\Omega,H)$, $p>0$, will be
denoted by $\L_p$.

Let $M$ be a $K$-valued square integrable
martingale. Further hypotheses on $M$ will be specified when
needed. For convenience, we shall say that $M$ satisfies hypothesis
(Q) if there exists a deterministic operator $Q \in \LL_1(K)$ such
that
\[
\langle\!\langle M,M \rangle\!\rangle(t) - \langle\!\langle M,M
\rangle\!\rangle(s) \leq (t-s)Q \qquad \forall 0 \leq s \leq t \leq T.
\]

Let $(Z,\mathcal{Z},m)$ be a $\sigma$-finite measure space, and $\mu$
a Poisson random measure on $Z \times [0,T]$ with compensator $m
\otimes \mathrm{Leb}$, where Leb stands for the Lebesgue measure on
$[0,T]$. The compensated measure $\mu-m\otimes\mathrm{Leb}$ will be
denoted by $\m$. We shall denote the space of functions $\phi:Z \to H$
such that $\|\phi\|_H \in L_p(Z,m)$, $p \geq 2$, by $L_p(Z)$.

For any $0 < t \leq T$, $\H_p(t)$ stands for the Banach space of
c\`adl\`ag adapted processes $u:\Omega \times [0,t] \to H$ such that
\[
\|u\|_{\H_p(t)} := \bigl( \E\sup_{s \leq t} \|u(s)\|^p \bigr)^{1/p} < \infty.
\]
We shall write $\H_p$ instead of $\H_p(T)$.

\medskip

Let us consider the equations
\begin{equation}
\label{eq:mumi}
du(t) + Au(t)\,dt + f(u(t))\,dt = B(u(t-))\,dM(t),
\qquad u(0)=u_0,
\end{equation}
and, for each $n\in\enne$,
\begin{equation}
\label{eq:troll}
du_n(t) + A_nu_n(t)\,dt + f_n(u_n(t))\,dt = B_n(u(t-))\,dM(t),
\qquad u(0)=u_{0n}.
\end{equation}
One has the following well-posedness result in $\H_2$. A proof (of a
more general result) can be found for instance in \cite{Kote-Doob}.
\begin{thm}
  Assume that $M$ satisfies hypothesis \emph{(Q)} and $u_0 \in
  \L_2(H)$. If $f \in \lip(H)$ and $B \in
  \lip\bigl(H,\LL_2(Q^{1/2}K,H)\bigr)$, then \eqref{eq:mumi} admits a
  unique mild solution $u \in \H_2$, which depends continuously on the
  initial datum $u_0$.
\end{thm}
\noindent
Clearly, if, for each $n\in\enne$, $(u_{0n},f_n,B_n)$ satisfy the same
type of assumptions, then \eqref{eq:troll} is also well-posed in
$\H_2$.

\medskip

Our first main result, whose proof is postponed to
\S\ref{ssec:prova1}, is the following.
\begin{thm}     \label{thm:nyo2} 
  Assume that $M$ satisfies hypothesis \emph{(Q)}.  Moreover, assume
  that
  \begin{itemize}\setlength{\itemsep}{0pt}
  \item[\emph{(i)}] for each $n \in \enne$, there exists $\eta_n \leq
    \eta$ such that $A_n+\eta_n I$ is a linear maximal monotone
    operator on $H$, and there exists
    $\lambda_0>0$ such that $(I+\lambda A_n)^{-1}h \to (I+\lambda
    A)^{-1}h$ as $n \to \infty$ for all $h \in H$ and
    $0<\lambda<\lambda_0$;
  \item[\emph{(ii)}] there exists a constant $L_f>0$ such that $f_n\in
    \lip(H)$ with $\|f_n\|_{\lip} + \|f\|_{\lip} \leq L_f$ for
    all $n \in \enne$ and $f_n \to f$ pointwise as $n \to \infty$;
  \item[\emph{(iii)}] there exists a constant $L_B>0$ such that $B_n
    \in \lip(H,\LL_2(Q^{1/2}K,H))$ with $\|B_n\|_{\lip} +
    \|B\|_{\lip} \leq L_B$ for all $n \in \enne$ and $B_n \to B$ pointwise
    as $n \to \infty$;
  \item[\emph{(iv)}] $u_{0n} \in \L_2$ for all $n\in\enne$ and $u_{0n}
    \to u_0$ in $\L_2$ as $n \to \infty$.
  \end{itemize}
  Let $u$ and $u_n$ be the mild solutions to \eqref{eq:mumi} and
  \eqref{eq:troll}, respectively. Then $u_n \to u$ in $\H_2$ as $n
  \to \infty$, that is
  \[
  \lim_{n \to \infty} \E\sup_{t\leq T}\,
    \bigl\| u_n(t) - u(t) \bigr\|^2 = 0.
  \]
\end{thm}
\begin{rmk}     \label{rmk:tanti}
  (a) The type of convergence of $A_n$ to $A$ assumed in (i) is also
  called \emph{convergence in the strong resolvent sense}.
  \smallskip\par\noindent 
  (b) Hypothesis (Q) is satisfied, for instance, if $M$ has stationary
  independent increments (in particular it $M$ is a L\'evy processes
  without drift, see e.g. \cite[p.~69]{PZ-libro}). One may remove this
  assumption at the price of assuming that $B$ satisfies a ``random''
  Lipschitz condition, i.e. a condition involving a predictable
  $\LL_1$-valued process rather than the (deterministic,
  time-independent) operator $Q$. Similarly, it would be possible to
  give a convergence result in $\H_p$, assuming that $B$ satisfies a
  different ``random'' Lipschitz condition involving the quadratic
  variation of $M$. We believe that these conditions are in general
  too difficult to check, and the corresponding results are of limited
  interest.
  \smallskip\par\noindent 
  (c) One could allow the coefficients $f$ and $B$ to depend also on
  $\omega \in \Omega$ and $t \in [0,T]$, assuming that they satisfy
  suitable measurability conditions and that their Lipschitz constants
  (with respct to the $H$-valued variable) do not depend on
  $(\omega,t)$. Details are left to the interested reader.
\end{rmk}

\medskip

We now turn to the case of equations driven by compensated Poisson
random measures. Consider the equations
\begin{equation}
\label{eq:mumip}
du(t) + Au(t)\,dt + f(u(t))\,dt = \int_Z G(z,u(t-))\,\m(dz,dt),
\qquad u(0)=u_0,
\end{equation}
and, for each $n\in\enne$,
\begin{equation}
\label{eq:trollp}
du_n(t) + A_nu_n(t)\,dt + f_n(u_n(t))\,dt = \int_Z G_n(z,u_n(t-))\,\m(dz,dt),
\qquad u(0)=u_{0n}.
\end{equation}
Recall that (see \cite{cm:JFA10}) if $f \in \lip(H)$ and
\[
\Bigl(\int_Z \|G(z,u)-G(z,v)\|^2\,m(dz)\Bigr)^{p/2}
+ \int_Z \|G(z,u)-G(z,v)\|^p\,m(dz) \lesssim \|u-v\|^p
\]
for all $u$, $v \in H$, then \eqref{eq:mumip} is well-posed in
$\H_p$. A completely analogous statement obviously holds for
\eqref{eq:trollp}.
Observing that
\[
\|\Phi\|^p_{L_2(Z)} + \|\Phi\|^p_{L_p(Z)} \lesssim_p 
\max\bigl(\|\Phi\|_{L_2(Z)},\|\Phi\|_{L_p(Z)}\bigr)^p
\lesssim_p \|\Phi\|^p_{L_2(Z)} + \|\Phi\|^p_{L_p(Z)}
\]
and recalling that one can turn the intersection of $L_2(Z)$ with
$L_p(Z)$ into a Banach space with the norm
\[
\| \cdot \|_{L_2(Z) \cap L_p(Z)} := \max\bigl(
\| \cdot \|_{L_2(Z)},\| \cdot \|_{L_p(Z)}\bigr)
\]
(see e.g.~\cite[p.~9]{KPS}), the above Lipschitz condition for $G$ can
be equivalently formulated as $h \mapsto G(\cdot,h) \in
\lip(H,L_2(Z) \cap L_p(Z))$.

\medskip

Our second main result is the following.
\begin{thm}     \label{thm:nyop} 
  Let $p \in [2,\infty[$. Assume that $A_n$ and $f_n$, $n\in\enne$,
  satisfy hypotheses \emph{(i)} and \emph{(ii)} of the previous
  theorem, and that
  \begin{itemize}
  \item[\emph{(iii')}] there exists a constant $L_G>0$ such that $h
    \mapsto G_n(\cdot,h) \in \lip\bigl(H,L_2(Z)\cap L_p(Z)\bigr)$ with
    $\|h \mapsto G_n(\cdot,h)\|_{\lip} + \|h \mapsto
    G(\cdot,h)\|_{\lip} \leq L_G$ for all $n \in \enne$ and
    \[
    \bnorm{G_n(\cdot,h)- G(\cdot,h)}_{L_2(Z)\cap L_p(Z)}
    \xrightarrow{n \to \infty} 0 \qquad \forall h \in H;
    \]
  \item[\emph{(iv')}] $u_{0n} \in \L_p$ for all $n\in\enne$ and $u_{0n}
    \to u_0$ in $\L_p$ as $n \to \infty$. 
  \end{itemize}
  Let $u$ and $u_n$ be the mild solutions to \eqref{eq:mumip} and
  \eqref{eq:trollp}, respectively. Then $u_n \to u$ in $\H_p$ as $n
  \to \infty$, that is
  \[
  \lim_{n \to \infty} \E\sup_{t\leq T}\,
    \bigl\| u_n(t) - u(t) \bigr\|^p = 0.
  \]
\end{thm}

\section{Preliminaries}     \label{sec:prel}
\subsection{Linear maximal monotone operators}
We are going to recall some definitions and (known) facts about linear
maximal (quasi-)monotone operators on Hilbert spaces, referring
e.g. to \cite{Bre-AF,Pazy} for details.

A linear operator $A:D(A) \subset H \to H$ is called maximal monotone
if $\ip{Ax}{x} \geq 0$ for all $x \in D(A)$ and $R(I+\lambda A)=H$ for
all $\lambda>0$. An operator $A$ is called maximal $\eta$-monotone if
$A+\eta I$ is maximal monotone.  Let $A$ be maximal $\eta$-monotone on
$H$ and, for $0<\lambda<1/\eta$, let $J_\lambda:=(I+\lambda A)^{-1}$
and $A_\lambda:=\lambda^{-1}(I-J_\lambda)$ (the latter operator is the
so-called Yosida regularization, or approximation, of $A$).  Then
\begin{itemize}
\item[(i)] $J_\lambda \in \LL(H)$ for all $0<\lambda<1/\eta$ with
  $\|J_\lambda x\| \leq (1-\lambda\eta)^{-1}\|x\|$ for all $x\in H$;
\item[(ii)] $A_\lambda \in \LL(H)$ for all $0<\lambda<1/\eta$ with
$\|A_\lambda x\| \leq (1-\lambda\eta)^{-1}\|Ax\|$ for all $x \in D(A)$;
\item[(iii)] $A_\lambda x = AJ_\lambda x$ for all $x \in H$;
\item[(iv)] $J_\lambda x \to x$ as $\lambda \to 0$ for all $x \in
  H$. In particular, by (iii), $A_\lambda x \to Ax$ as $\lambda \to 0$
  for all $x \in D(A)$.
\end{itemize}
It should be noted that the above properties of the resolvent
$J_\lambda$ and of the Yosida approximation $A_\lambda$ continue to
hold, \emph{mutatis mutandis}, for the much more general class of
nonlinear (quasi-)$m$-accretive operators on Banach spaces (see
e.g. \cite{Barbu}).

\medskip

We shall need for the proofs of the main results the following
inhomogeneous version of the Trotter-Kato's theorem.
\begin{thm}     \label{thm:titikaka}
  Let $A$ and $A_n$, $n\in\enne$, be maximal monotone operators on
  $H$; $f$ and $f_n$, $n\in\enne$, be elements of $L_1([0,T],H)$; $u_0$
  and $u_{0n}$, $n\in\enne$, be elements of $H$. Let $u$ and $u_n$
  denote the mild solutions to the equations
  \begin{gather*}
    u' + Au = f, \qquad u(0) =u_0,\\
    u_n' + A_nu_n = f_n, \qquad u_n(0) =u_{0n},
  \end{gather*}
  respectively. Suppose that, as $n \to \infty$, $A_n \to A$ in the
  strong resolvent sense, $u_{0n} \to u$ in $H$ and $f_n \to f$ in
  $L^1([0,T],H)$. Then
  \[
  \lim_{n\to\infty} \sup_{t\leq T} \|u_n(t)-u(t)\| = 0.
  \]
\end{thm}
\begin{proof}
  See e.g. \cite[p.~241]{Barbu} for a proof (of a much more general
  result) that uses the theory of $m$-accretive operators on Banach
  spaces, or \cite{KvN1} for a ``linear'' proof using the
  factorization method.
\end{proof}

\subsection{Stochastic integration with jumps and maximal
  inequalities}
We shall use the theory of stochastic integration with respect to
Hilbert space-valued martingales, about which we refer to \cite{Met}
for a detailed treatment. Here we shall essentially limit ourselves to
fixing notation.

For a $K$-valued square integrable martingale $M$, let $Q_M$ be
the unique $\LL_1(K)$-valued predictable process $Q_M$ such that
\[
\langle\!\langle M,M \rangle\!\rangle(t) = \int_0^t Q_M(s)\,
d\langle M,M \rangle(s).
\]
We shall denote by $\Lambda^2_M(K,H)$ the closure of the space of
$\LL(K,H)$-valued simple process in the space of processes $\Phi$ whose
values are linear (possibly unbounded) operator from $K$ to $H$ such
that $\Phi(t)Q_M^{1/2}(t) \in \LL_2(K,H)$ for all $(\omega,t) \in \Omega
\times [0,T]$, $\Phi Q_M^{1/2}h$ is predictable for all $h \in H$, and
\[
\E \int_0^T \bigl\|\Phi(t)Q_M^{1/2}(t) \bigr\|^2_{\LL_2(K,H)}\,
d\langle M,M \rangle(t) < \infty.
\]
For any $\Phi \in \Lambda^2_M(K,H)$, the stochastic integral $\Phi
\cdot M$ is an $H$-valued square integrable martingale with
$\bigl\langle \Phi \cdot M,\Phi \cdot M \bigr\rangle = \bigl\|\Phi
Q_M^{1/2} \bigr\|^2_{\LL_2} \cdot \langle M,M \rangle$. Note that, if
$M$ satisfies the (Q) hypothesis, then
\[
\E\int_0^T \bigl\|\Phi(t)Q_M^{1/2}(t)\bigr\|^2_{\LL_2(K,H)}\,
d\langle M,M \rangle(t) \leq 
\E\int_0^T \bigl\|\Phi(t)Q^{1/2}\bigr\|^2_{\LL_2(K,H)}\,
dt.
\]

\medskip

In the following proposition we collect some (known) maximal
inequalities for stochastic convolutions driven by martingales, of
which we sketch a proof for the reader's convenience. More details can
be found e.g. in \cite{HauSei2}.
\begin{prop}
  \label{prop:yayo}
  Let $B$ be a process taking values in the space of linear (not
  necessarily bounded) operators from $K$ to
  $H$, and set
  \[
  Y(t) := \int_0^t S(t-s)B(s)\,dM(s), \qquad 0 \leq t \leq T.
  \]
  The following holds true:
  \begin{itemize}
  \item[\emph{(i)}] if $B \in \Lambda^2_M(K,H)$, then $Y \in \H_2$ and
    \begin{equation}
      \label{eq:maxi2}
      \|Y\|^2_{\H_2} \equiv \E\sup_{t \leq T} \|Y(t)\|^2 \lesssim_\eta 
      \E\int_0^T \|B(t)Q_M^{1/2}(t)\|^2_{\LL_2(K,H)} \,d\langle M,M
      \rangle(t);
    \end{equation}
  \item[\emph{(ii)}] if $B:[0,T] \times \Omega \to \LL(K,H)$ is
    predictable and there exists $p \in [2,\infty[$ such that the
    right-hand side of \eqref{eq:maxip} below is finite, then $Y \in
    \H_p$ and
    \begin{equation}
      \label{eq:maxip}
      \|Y\|^p_{\H_p} \equiv 
      \E\sup_{t \leq T} \|Y(t)\|^p \lesssim_{p,\eta} \E\Bigl(\int_0^T
      \|B(t)\|^2_{\LL(K,H)}\,d[M,M](t)\Bigr)^{p/2}.
    \end{equation}
  \end{itemize}
\end{prop}
\begin{proof}
  Let $S^\eta(t):=e^{-\eta t}S(t)$, $t \geq 0$. Then $S^\eta$ is a
  strongly continuous contraction semigroup, and, by Sz.-Nagy's
  dilation theorem, there exist a (separable) Hilbert space $\bar{H}
  \supset H$ and a unitary strongly continuous group
  $(U(t))_{t\in\erre}$ on $\bar{H}$ such that
  \[
  S^\eta(t) = \pi \circ U(t) \circ i \qquad \forall t \geq 0,
  \]
  where $i:H \to \bar{H}$ is an isometric embedding and $\pi: \bar{H}
  \to H$ is an orthogonal projection. We thus have
  \begin{align*}
    &\E\Big\| \int_0^t S(t-s)B(s)\,dM(s) \Big\|^2
    \leq e^{2\eta T}\E\Big\| \int_0^t S^\eta(t-s)B(s)\,dM(s) \Big\|^2\\
    &\qquad \lesssim_\eta \E\Big\| U(t)\int_0^t U(-s)B(s)\,dM(s)
    \Big\|^2_{\bar{H}}
    \leq \E\Big\| \int_0^t U(-s)B(s)\,dM(s) \Big\|^2_{\bar{H}}\\
    &\qquad \leq \E\int_0^t \|B(s)\|^2_{Q_M}\,d\langle M,M \rangle(s).
  \end{align*}
  Then \eqref{eq:maxi2} follows by Doob's inequality for real-valued
  submartingales. The proof of \eqref{eq:maxip} is completely
  analogous: it follows from Burkholder's inequality, rather than from
  the isometric property of the stochastic integral with respect to
  $M$, taking into account the easy estimate $[B\cdot M, B\cdot M]
  \leq \|B\|^2_\LL \cdot [M,M]$.
\end{proof}
\begin{rmk}
  Unfortunately it is not possible to replace the operator norm of $B$
  in \eqref{eq:maxip} with the Hilbert-Schmidt norm of $BQ_M^{1/2}$,
  cf.~e.g.~\cite{HauSei2} for a brief discussion of this issue.
\end{rmk}

\medskip

We shall also need a maximal inequality for stochastic convolution
with respect to compensated Poisson random measures obtained in
\cite{cm:JFA10}. Here $\mathcal{P}$ stands for the predictable
$\sigma$-field.
\begin{prop}[\cite{cm:JFA10}, Prop.~3.3]
  Assume that $G:\Omega \times [0,T] \times Z \to H$ is $\mathcal{P}
  \otimes \mathcal{Z}$-measurable and there exists $p \in [2,\infty[$
  such that the right-hand side in \eqref{eq:**} below is
  finite. Then, setting
  \[
  Y(t) := \int_0^t\!\int_Z S(t-s) G(s,z)\,\m(ds,dz), \qquad 0 \leq t \leq T,
  \]
  one has $Y \in \H_p$ and
  \begin{equation}       \label{eq:**}
    \begin{aligned}
    \|Y\|^p_{\H_p} &\equiv \E\sup_{t \leq T} \|Y(t)\|^p\\
    &\lesssim_{p,\eta} \E \int_0^T \biggl[
    \int_Z \|G(t,z)\|^p\,m(dz) + \Bigl( \int_Z \|G(t,z)\|^2\,m(dz)
    \Bigr)^{p/2} \biggr]\,dt.
    \end{aligned}
  \end{equation}
\end{prop}
Note that inequalities \eqref{eq:maxi2}, \eqref{eq:maxip} and
\eqref{eq:**} can equivalently be written as
\begin{align}
  \tag{\ref{eq:maxi2}'}
  \|Y\|_{\H_2} &\lesssim \norm{\bigl(
    \|BQ_M^{1/2}\|_{\LL_2}^2\cdot\langle M,M\rangle \bigr)^{1/2}}_{\L_2},\\
  \tag{\ref{eq:maxip}'}
  \|Y\|_{\H_p} &\lesssim \norm{\bigl(
    \|B\|^2_{\LL} \cdot [M,M] \bigr)^{1/2}}_{\L_p},\\
  \tag{\ref{eq:**}'}
  \|Y\|_{\H_p} &\lesssim \bigl\| G
  \bigr\|_{L_p(\Omega\times[0,T],L_2(Z)\cap L_p(Z))}.
\end{align}

\begin{rmk}
  (i) A corresponding inequality for stochastic integrals and
  convolutions with respect to L\'evy processes was established in
  \cite{cm:MF10}. An analogous estimate holds if the Hilbert space $H$
  is replaced by an $L_q$ space (see \cite{cm:EJP10} for a basic result, and
  \cite{Dirks:th} for far-reaching generalizations).

  (ii) The maximal estimates of the previous two propositions continue
  to hold in the case that $A$ has a bounded $H^\infty$-calculus of
  angle less than $\pi/2$. In fact, exactly the same proofs go
  through, using a different (and more sophisticated) dilation
  theorem, cf. e.g. \cite{VerWei} for the case of stochastic
  convolutions in UMD Banach spaces of type 2 with respect to a Wiener
  process. In the context of Hilbert spaces, however, the classes of
  quasi-monotone operators and of operators with bounded
  $H^\infty$-calculus mentioned above essentially coincide (see
  \cite{LeM03} for a precise result).
\end{rmk}

\section{Convergence of stochastic convolutions I}     
\label{sec:conv1}
Throughout this and the following section we assume that $\eta=0$, in
particular that $A$ is maximal monotone, rather than just maximal
quasi-monotone. That this comes at no loss of generality is showed in
Remark \ref{rmk:noloss} below.

Let us consider the linear stochastic evolution equation on $H$ 
\begin{equation}     \label{eq:lse}
dy(t) + Ay(t)\,dt = B(t)\,dM(t), \qquad y(0)=y_0,
\end{equation}
whose mild solution is \emph{defined}, at least formally, as
\[
y(t) = S(t)y_0 + \int_0^t S(t-s)B(s)\,dM(s).
\]
It is immediate that $y_0 \in \L_2$, $B \in \Lambda^2_M$ imply $y \in
\H_2$, and that, for any $p \in [2,\infty[$, $y_0 \in \L_p$, $B$
satisfying the hypotheses of Proposition \ref{prop:yayo}(ii) imply $y
\in \H_p$.

In the first of the following two subsections we establish convergence
to $y$, in $\H_2$ and in $\H_p$, of the solutions to the equations
obtained replacing $A$ in \eqref{eq:lse} with its Yosida
regularization. In the second subsection we consider, more generally,
the equations obtained replacing $A$ by $A_n$, with $A_n$ converging
to $A$ in the strong resolvent sense.

\subsection{Yosida approximation of $A$}
Let $A_\lambda$, $\lambda>0$, be the Yosida approximation of $A$, and
consider the regularized equation
\begin{equation}
  \label{eq:yol}
  dy_\lambda(t) + A_\lambda y_\lambda\,dt = B(t)\,dM(t), \qquad
  y_\lambda(0)=y_0,
\end{equation}
whose mild solution can be written, formally for the time being, as
\[
y_\lambda(t) = e^{-tA_\lambda}y_0 + \int_0^t e^{-(t-s)A_\lambda}B(s)\,dM(s).
\]
In analogy to the case of equation \eqref{eq:lse}, $y_0 \in \L_2$ and
$B \in \Lambda^2_M$ imply that $y_\lambda \in \H_2$, while $y_0 \in
\L_p$ and $B$ predictable with $\bigl(\|B\|^2_{\LL}\cdot
[M,M]\bigr)^{1/2} \in \L_p$ imply that $y_\lambda \in \H_p$.

\medskip

We start with an elementary convergence result which will be needed in
the proof of Theorem \ref{thm:yo2sc} below.
\begin{lemma}
  Assume that $y_0 \in \L_2$ and $B$ satisfies assumption (i) or (ii)
  of Proposition \ref{prop:yayo}. Let $y$ and $y_\lambda \in \H_2$ be
  the solutions to \eqref{eq:lse} and \eqref{eq:yol},
  respectively. Then one has $y_\lambda(t) \to y(t)$ in $\L_2$ for all
  $t \in [0,T]$ as $\lambda \to 0$, i.e.
  \[
  \lim_{\lambda \to 0} \E\bigl\| y_\lambda(t)-y(t) \bigr\|^2 = 0
  \qquad \forall t \in [0,T].
  \]
\end{lemma}
\begin{proof}
  One has
  \begin{equation}
    \label{eq:pfff}
    \begin{aligned}
      \E\bigl\| y_\lambda(t)-y(t) \bigr\|^2 &\lesssim
      \E\bigl\| e^{-tA_\lambda}y_0 - S(t)y_0 \bigr\|^2\\
      &\quad + \E\Bigl\| \int_0^t
      \bigl(e^{-(t-s)A_\lambda}B(s)-S(t-s)B(s)\bigr)\,dM(s) \Bigr\|^2.
    \end{aligned}
  \end{equation}
  Let us assume first that $B \in \Lambda^2_M(K,H)$. Recall that, by
  the Trotter-Kato's theorem (see e.g.~\cite[p.~88]{Pazy}), one has
  $e^{-tA_\lambda}h \to S(t)h$ as $\lambda \to 0$ for all $t \in
  [0,T]$ and all $h \in H$.  By the isometric property of the
  stochastic integral with respect to $M$, the second term on the
  right-hand side of the above inequality is equal to
  \begin{multline*}
  \E\int_0^t \bigl\| \bigl(
  e^{-(t-s)A_\lambda}B(s)-S(t-s)B(s)\big)Q_M(s)^{1/2} \bigr\|^2_{\LL_2}
  \,d\langle M,M \rangle(s)\\
  = \E\int_0^t \sum_{j=1}^\infty \bigl\|
  e^{-(t-s)A_\lambda}B(s)Q_M^{1/2}(s)e_j - S(t-s)B(s)Q_M^{1/2}e_j \bigr\|^2
  \,d\langle M,M \rangle(s),
  \end{multline*}
  where $(e_j)_{j\in\enne}$ is an orthonormal basis of $K$. Then one has
  \[
  \lim_{\lambda \to 0} \big\|
  e^{-(t-s)A_\lambda}B(s)Q_M^{1/2}(s)e_j - S(t-s)B(s)Q_M^{1/2}(s)e_j 
  \big\|^2 = 0
  \]
  for all $s \leq t$ and for all $j\in\enne$. Since the operator norms
  of $S(t)$ and $e^{-tA_\lambda}$ are not larger than one for all $t
  \in [0,T]$, one also has
  \[
  \big\| e^{-(t-s)A_\lambda}B(s)Q_M^{1/2}(s)e_j -
  S(t-s)B(s)Q_M^{1/2}(s)e_j \big\|^2 \lesssim \big\|
  B(s)Q_M^{1/2}(s)e_j \big\|^2
  \]
  for all $0 \leq s \leq t$, and
  \begin{multline*}
    \E\int_0^t \sum_{j=1}^\infty \big\| B(s)Q_M^{1/2}(s)e_j \big\|^2\,
    d\langle M,M \rangle(s)\\
    = \E\int_0^t \big\| B(s)Q_M^{1/2}(s) \big\|^2_{\LL_2(K,H)}\,
    d\langle M,M \rangle(s)<\infty.
  \end{multline*}
  The dominated convergence theorem then implies that the second term
  on the right-hand side of \eqref{eq:pfff} tends to zero as $\lambda
  \to 0$. A completely analogous (but simpler) argument shows that the
  same is true for the first term on the r.h.s. of \eqref{eq:pfff}.

  If $B$ satisfies the assumptions of Proposition \ref{prop:yayo}(ii)
  for some $p \geq 2$, it certainly does for $p=2$, in which case we
  have
\begin{align*}
\E \bigl( \|BQ_M^{1/2}\|^2_{\LL_2} \cdot \langle M,M\rangle \bigr)
&\leq \E\bigl( \bigl( \|B\|^2_\LL \|Q_M^{1/2}\|^2_{\LL_2}\bigr) \cdot 
     \langle M,M\rangle \bigr)\\
&\leq \E\bigl( \|B\|^2_\LL \cdot \langle M,M\rangle \bigr)
= \E\bigl( \|B\|^2_\LL \cdot [M,M] \bigr) < \infty,
\end{align*}
where we have used the ideal property of the space of Hilbert-Schmidt
operators, the identity $\|Q_M^{1/2}\|^2_{\LL_2}=\|Q_M\|_{\LL_1}$, and
the fact that $\operatorname{Tr} Q_M \leq 1$. We have thus shown that
$B \in \Lambda^2_M$, a condition which has already been proved to
imply the claim.
\end{proof}

\begin{thm}    \label{thm:yo2sc}
  Assume that $y_0 \in \L_2$ and $B \in \Lambda^2_M(K,H)$.  Let $y$
  and $y_\lambda \in \H_2$ be the solutions to \eqref{eq:lse} and
  \eqref{eq:yol}, respectively. Then one has $y_\lambda \to y$ in
  $\H_2$ as $\lambda \to 0$, i.e.
  \begin{equation}     \label{eq:conv}
    \lim_{\lambda \to 0} \E\sup_{t\leq T}
    \big\| y_\lambda(t)-y(t) \big\|^2 = 0.
  \end{equation}
\end{thm}
\begin{proof}
  Let us introduce two auxiliary regularized equations as follows:
  \begin{alignat}{2}
    \label{eq:reg1}
    d\ep{y}(t) + A\ep{y}(t)\,dt &=
    B^\varepsilon(t)\,dM,&
    \qquad \ep{y}(0) &= y_0^\varepsilon,\\
    \label{eq:reg2}
    dy_\lambda^\varepsilon(t) + A_\lambda y_\lambda^\varepsilon(t)\,dt
    &= B^\varepsilon(t)\,dM,& \qquad y_\lambda^\varepsilon(0) &=
    y_0^\varepsilon,
  \end{alignat}
  with $\ep{y}_0:=(I+\varepsilon A)^{-1}y_0$, $\ep{B} :=
  (I+\varepsilon A)^{-1}B$, for $\varepsilon>0$.
  The triangle inequality yields
  \begin{equation} 
    \label{eq:tre} 
    \| y-y_\lambda \|_{\H_2} \leq \|
    y-\ep{y} \|_{\H_2} + \| \ep{y} -
    \ep{y}_\lambda \|_{\H_2} + \| \ep{y}_\lambda -
    y_\lambda \|_{\H_2}.
  \end{equation}
  By Proposition \ref{prop:yayo}(i) one gets
  \begin{align*}
    \E\sup_{t \leq T} \big\| y(t)-\ep{y}(t) \big\|^2 &\lesssim
    \E\|y_0-\ep{y}_0\|^2 +
    \E\int_0^T \big\|(B(s)-\ep{B}(s))Q_M^{1/2}(s)\big\|^2_{\LL_2}
      \,d\langle M,M\rangle(s),\\
    \E\sup_{t \leq T} \big\| y_\lambda(t)-\ep{y}_\lambda(t) \big\|^2
    &\lesssim \E\|y_0-\ep{y}_0\|^2 + \E\int_0^T
    \big\|(B(s)-\ep{B}(s))Q_M^{1/2}(s)\big\|^2_{\LL_2}\,d\langle M,M\rangle(s).
  \end{align*}
  Since $(I+\lambda A)^{-1}$ is contracting, the dominated convergence
  theorem implies that the right-hand sides of the above inequalities
  converge to zero as $\varepsilon \to 0$. Let us fix $\delta>0$. Then
  there exists $\varepsilon>0$ such that
  \[
  \| y-\ep{y} \|_{\H_2} + \| \ep{y}_\lambda -
  y_\lambda\|_{\H_2} < \frac12 \delta
  \]
  for all $\lambda>0$. We shall keep $\varepsilon$ fixed from now on.
  In order to conclude the proof we have to show that $\| \ep{y} -
  \ep{y}_\lambda \|_{\H_2} < \delta/2$ for $\lambda$ small enough.
  To this purpose note that, for any $\lambda>0$, $\ep{y}_\lambda$ is
  a strong solution (not just a mild solution) to \eqref{eq:reg2}
  because $A_\lambda$ is a bounded operator. Therefore, for any
  $\lambda$, $\mu>0$, we infer that $\ep{y}_\lambda-\ep{y}_\mu$ is a
  strong solution to the deterministic evolution equation
  \begin{equation*}
  (\ep{y}_\lambda-\ep{y}_\mu)' + A_\lambda \ep{y}_\lambda - A_\mu
  \ep{y}_\mu = 0, \qquad \ep{y}_\lambda(0) - \ep{y}_\mu(0)=0.
  \end{equation*}
  Taking the scalar product of both sides with
  $\ep{y}_\lambda-\ep{y}_\mu$, one has
  \[
  \frac12 \frac{d}{dt} \|\ep{y}_\lambda-\ep{y}_\mu\|^2 +\ip{A_\lambda
    \ep{y}_\lambda - A_\mu \ep{y}_\mu}{\ep{y}_\lambda-\ep{y}_\mu} = 0.
  \]
  Recalling the identity $\lambda A_\lambda = I-J_\lambda$, one has
  \begin{align*}
    \ep{y}_\lambda - \ep{y}_\mu &= (J_\lambda \ep{y}_\lambda - J_\mu
    \ep{y}_\mu) + (\ep{y}_\lambda - J_\lambda \ep{y}_\lambda)
    - (\ep{y}_\mu - J_\mu \ep{y}_\mu) \\
    &= (J_\lambda \ep{y}_\lambda - J_\mu \ep{y}_\mu) + \lambda
    A_\lambda \ep{y}_\lambda -\mu A_\mu \ep{y}_\mu.
  \end{align*}
  This yields, thanks to the identity $A_\lambda = AJ_\lambda$,
  \begin{align*}
    \ip{A_\lambda \ep{y}_\lambda - A_\mu
      \ep{y}_\mu}{\ep{y}_\lambda-\ep{y}_\mu} &= \ip{AJ_\lambda
      \ep{y}_\lambda - AJ_\mu \ep{y}_\mu}{J_\lambda \ep{y}_\lambda
      - J_\mu \ep{y}_\mu}\\
    &\quad + \ip{A_\lambda \ep{y}_\lambda - A_\mu \ep{y}_\mu}
    {\lambda A_\lambda \ep{y}_\lambda -\mu A_\mu \ep{y}_\mu}\\
    &\geq \ip{A_\lambda \ep{y}_\lambda - A_\mu \ep{y}_\mu} {\lambda
      A_\lambda \ep{y}_\lambda -\mu A_\mu \ep{y}_\mu},
  \end{align*}
  thus also
  \begin{equation}
  \label{eq:lm2}
  \|\ep{y}_\lambda - \ep{y}_\mu\|^2(t) \lesssim (\lambda+\mu) \int_0^t
  \big( \|A_\lambda \ep{y}_\lambda(s)\|^2 + \|A_\mu \ep{y}_\mu(s)\|^2
  \big)\,ds,
  \end{equation}
  and
  \[
  \E \sup_{t\leq T} \|\ep{y}_\lambda - \ep{y}_\mu\|^2(t) \lesssim
  (\lambda + \mu) \E \int_0^T \big( \|A_\lambda \ep{y}_\lambda(s)\|^2
  + \|A_\mu \ep{y}_\mu(s)\|^2 \big)\,ds.
  \]
  Since it holds
  \[
  \ep{y}_\lambda(s) = e^{-tA_\lambda}\ep{y}_0 + \int_0^t e^{-(t-s)A_\lambda}
  \ep{B}(s)\,dM(s),
  \]
  recalling that $\|A_\lambda x\| \leq \|Ax\|$ for all $x \in D(A)$,
  one has
  \begin{align*}
    \E \|A_\lambda \ep{y}_\lambda(s)\|^2 &\lesssim \E\big\|
    e^{-sA_\lambda} A_\lambda \ep{y}_0 \big\|^2 + \E\Big\|
    \int_0^s e^{-(s-r)A_\lambda} A_\lambda \ep{B}(r)\,dM(r) \Big\|^2\\
    &\leq \E \|Ay_0^\varepsilon\|^2 + \E\int_0^s
    \bigl\|AB^\varepsilon(r)Q_M^{1/2}(r)\bigr\|_{\LL_2}^2\,d\langle M,M\rangle(r),
  \end{align*}
  which implies, taking into account that $A(I+\varepsilon
  A)^{-1}=A_\varepsilon$ is a bounded operator,
  \begin{align*}
    \E \sup_{t\leq T} \|\ep{y}_\lambda - \ep{y}_\mu\|^2(t) &\lesssim
    T(\lambda + \mu) \Big( \E \|Ay_0^\varepsilon\|^2
    + \E\int_0^T \bigl\|AB^\varepsilon(r)Q_M^{1/2}(r)\bigr\|_{\LL_2}^2
        \,d\langle M,M\rangle(r) \Big)\\
    &\lesssim_\varepsilon T(\lambda + \mu) \Big( \E \|y_0\|^2 +
    \E\int_0^T \bigl\|B(r)Q_M^{1/2}(s)\bigr\|_{\LL_2}^2\,d\langle M,M\rangle(r)
    \Big),
  \end{align*}
  i.e. $\lambda \mapsto \ep{y}_\lambda$ is a Cauchy net in
  $\mathbb{H}_2$. In particular, there exist $\ep{z}$ such that
  $\ep{y}_\lambda \to \ep{z}$ in $\mathbb{H}_2$ as $\lambda \to
  0$. Clearly this implies that $\ep{y}_\lambda(t) \to \ep{z}(t)$ in
  $\mathbb{L}_2$ for all $t\in[0,T]$ as $\lambda \to 0$. Since the
  previous lemma implies that it also holds $\ep{y}_\lambda(t) \to
  \ep{y}(t)$ in $\mathbb{L}_2$ for all $t\in[0,T]$ as $\lambda \to 0$,
  we infer that $\ep{z}(t)=\ep{y}(t)$ for all $t \in [0,T]$. Then one
  has
  \[
  \| \ep{y}_\lambda - \ep{y} \|_{\H_2} = \big( \E \sup_{t\leq T}
  \|\ep{y}_\lambda - \ep{y}\|^2(t) \big)^{1/2} \lesssim_{\varepsilon}
  \sqrt{T\lambda},
  \]
  which is obviously bounded above by $\delta/2$ for $\lambda$ small
  enough.
\end{proof}

An analogous result holds in the case $p>2$, adapting the
assumptions on the coefficient $B$.
\begin{thm}     \label{thm:yopsc}
  Let $p>2$. Assume that $B$ satisfies the hypotheses of Proposition
  \ref{prop:yayo}(ii). Then one has $y_\lambda \to y$ in $\H_p$ as
  $\lambda \to 0$, i.e.
  \[
  \lim_{\lambda \to 0} \E\sup_{t \leq T} \|y_\lambda(t)-y(t)\|^p = 0,
  \]
  where $y_\lambda$ and $y$ denote the mild solutions to \eqref{eq:lse}
  and \eqref{eq:yol}, respectively.
\end{thm}
\begin{proof}
  In analogy to the argument used in the proof of the previous
  theorem, one has
  \begin{equation} 
    \label{eq:trep}
    \| y-y_\lambda \|_{\H_p} \leq \|
    y-\ep{y} \|_{\H_p} + \| \ep{y} -
    \ep{y}_\lambda \|_{\H_p} + \| \ep{y}_\lambda -
    y_\lambda \|_{\H_p},
  \end{equation}
  as well as, by Proposition \ref{prop:yayo}(ii),
  \begin{align*}
    \E\sup_{t \leq T} \big\| y(t)-\ep{y}(t) \big\|^p &\lesssim
    \E\|y_0-\ep{y}_0\|^p +
    \E\Bigl(\int_0^T \big\|B(s)-\ep{B}(s)\big\|^2_\LL\,d[M,M](s)\Bigr)^{p/2},\\
    \E\sup_{t \leq T} \big\| y_\lambda(t)-\ep{y}_\lambda(t) \big\|^p
    &\lesssim \E\|y_0-\ep{y}_0\|^p + \E\Bigl(\int_0^T
    \big\|B(s)-\ep{B}(s)\big\|^2_\LL\,d[M,M](s)\Bigr)^{p/2}.
  \end{align*}
  Let $\delta>0$ be arbitrary but fixed. Then there exists
  $\varepsilon>0$ such that
  \[
  \| y-\ep{y} \|_{\H_p} + \| \ep{y}_\lambda -
  y_\lambda\|_{\H_p} < \frac12 \delta
  \]
  for all $\lambda>0$. Keeping $\varepsilon$ fixed from now on, let us
  show that $\|\ep{y}-\ep{y}_\lambda\|_{\H_p}<\delta/2$ for $\lambda$
  sufficiently small. As in the proof of the previous theorem,
  exploiting the monotonicity of $A$ and using properties of the
  Yosida approximation, we arrive at
  \[
  \|\ep{y}_\lambda - \ep{y}_\mu\|^2(t) \lesssim (\lambda+\mu) \int_0^t
  \big( \|A_\lambda \ep{y}_\lambda(s)\|^2 + \|A_\mu \ep{y}_\mu(s)\|^2
  \big)\,ds.
  \]
  Raising both sides to the $p/2$-th power and appealing to H\"older's
  inequality we obtain
  \[
  \|\ep{y}_\lambda(t) - \ep{y}_\mu(t)\|^p \lesssim_T (\lambda+\mu)^{p/2}
  \int_0^t
  \big( \|A_\lambda \ep{y}_\lambda(s)\|^p + \|A_\mu \ep{y}_\mu(s)\|^p
  \big)\,ds,
  \]
  hence also
  \[
  \E \sup_{t\leq T} \|\ep{y}_\lambda(t) - \ep{y}_\mu(t)\|^p \lesssim_T
  (\lambda + \mu)^{p/2} \E \int_0^T \big( \|A_\lambda \ep{y}_\lambda(s)\|^p
  + \|A_\mu \ep{y}_\mu(s)\|^p \big)\,ds.
  \]
  Note that one has
  \begin{align*}
    \E \|A_\lambda \ep{y}_\lambda(s)\|^p &\lesssim \E\big\|
    e^{-sA_\lambda} A_\lambda \ep{y}_0 \big\|^p + \E\Bigl\|
    \int_0^s e^{-(s-r)A_\lambda} A_\lambda \ep{B}(r)\,dM(r) \Bigr\|^p\\
    &\leq \E \|Ay_0^\varepsilon\|^p + \E\Bigl(\int_0^s
    \|AB^\varepsilon(r)\|_\LL^2\,d[M,M](r)\Bigr)^{p/2},
  \end{align*}
  which implies
  \[
  \E \sup_{t\leq T} \|\ep{y}_\lambda(t) - \ep{y}_\mu(t)\|^p
  \lesssim_{T,\varepsilon}(\lambda + \mu)^{p/2} \left( \E \|y_0\|^p +
  \E\Bigl(\int_0^T \|B(r)\|_\LL^2\,d[M,M](r)\Bigr)^{p/2} \right).
  \]
  This shows that $\lambda \mapsto \ep{y}_\lambda$ is a Cauchy net in
  $\mathbb{H}_p$, from which we infer that there exists a constant
  $N$, depending on $\varepsilon$ and $T$, such that
  \[
  \| \ep{y}_\lambda - \ep{y} \|_{\H_p} \leq N \sqrt{\lambda},
  \]
  the right-hand side of which is clearly bounded above by $\delta/2$
  for $\lambda$ small enough.
\end{proof}

The estimates contained in the following corollary are simply
extracted from the proofs of the previous two theorems. Since they will
be used in the next subsection, we state them explicitly for clarity
of exposition.
\begin{coroll}     \label{cor:utile}
  Under the assumptions of Theorem \ref{thm:yo2sc}, the following
  inequality holds, for any $\varepsilon>0$ and $\lambda>0$:
  \begin{align*}
    \norm{y-y_\lambda}^2_{\H_2} &\lesssim \E\|y_0-\ep{y}_0\|^2 
    + \E\int_0^T \|B(s)-\ep{B}(s)\|^2_Q\,d\langle M,M\rangle(s)\\
    &\quad + T\lambda \Big( \E\|A\ep{y}_0\|^2
    + \E\int_0^T \|A\ep{B}(s)\|^2_Q\,d\langle M,M\rangle(s) \Big).
  \end{align*}
  Assume that there exists $p>2$ such that the assumptions of Theorem
  \ref{thm:yopsc} are satisfied. Then one has, for any $\varepsilon>0$
  and $\lambda>0$,
  \begin{align*}
    \norm{y-y_\lambda}^p_{\H_p} &\lesssim \E\|y_0-\ep{y}_0\|^p 
    + \E\Bigl(\int_0^T \|B(s)-\ep{B}(s)\|^2_\LL\,d[M,M](s)\Bigr)^{p/2}\\
    &\quad + (T\lambda)^{p/2} \biggl( \E\|A\ep{y}_0\|^p
    + \E\Bigl(\int_0^T \|A\ep{B}(s)\|^p_\LL\,d[M,M](s)\Bigr)^{p/2} \biggr).
  \end{align*}
\end{coroll}

\subsection{Approximation of $A$ in the strong resolvent sense}
For any $n \in \enne$, let $A_n$ be a linear maximal monotone operator
on $H$, and denote by $S_n$ the strongly continuous semigroup of
contractions generated by $-A_n$. Then the mild solution to the
equation
\begin{equation}     \label{eq:ressa}
dy_n(t) + A_ny_n(t)\,dt = B(t)\,dM(t), \qquad y_n(0)=y_0,
\end{equation}
defined by the variation of constants formula as
\[
y_n(t) = S_n(t)y_0 + \int_0^t S_n(t-s)B(s)\,dM(s),
\]
is well-defined as a process in $\H_2$ or in $\H_p$, under the
measurability and integrability assumptions on $B$ of Proposition
\ref{prop:yayo}(i) or (ii), respectively.

We start with a generalization of Theorem \ref{thm:yo2sc}.
\begin{thm}     \label{thm:nyo2sc}
  Assume that $A_n \to A$ in the strong resolvent sense as $n \to
  \infty$ and that the hypotheses on $B$ of Proposition
  \ref{prop:yayo}(i) are met. Then one has
  $y_n \to y$ in $\H_2$ as $n \to \infty$, i.e.
  \[
  \lim_{n \to \infty} \E\sup_{t\leq T}\|y_n(t)-y(t)\|^2 = 0.
  \]
\end{thm}
\begin{proof}
  Let us denote by $A_{n\lambda}:=A_n(I+\lambda A_n)^{-1}$,
  $\lambda>0$, the Yosida approximation of $A_n$, and consider the
  regularized equations
  \begin{gather}
  \label{approx13}
    dy_\lambda(t)+A_\lambda y_\lambda(t)\,dt= B(t)\,dM(t),
    \qquad y_\lambda(0)=y_0,\\
  \label{approx23}
    dy_{n\lambda}(t)+A_{n\lambda} y_{n\lambda}(t)\,dt= B(t)\,dM(t),
    \qquad y_{n\lambda}(0)=y_0.
  \end{gather} 
  By the triangle inequality one has
  \begin{equation}
    \label{triang2}
    \norm{y-y_n}_{\H_2} \leq \norm{y-y_\lambda}_{\H_2}
    + \norm{y_\lambda-y_{n\lambda}}_{\H_2}
    + \norm{y_{n\lambda}-y_n}_{\H_2}.
  \end{equation}
  By Corollary \ref{cor:utile} we infer that, for any
  $\varepsilon>0$ and $\lambda>0$, the following inequalities hold
  true:
  \begin{gather*}
  \begin{split}
    \norm{y-y_\lambda}^2_{\H_2} &\lesssim \E\|y_0-\ep{y}_0\|^2 
    + \E\int_0^T \bnorm{(B(s)-\ep{B}(s))Q_M^{1/2}(s)}^2_{\LL_2}
        \,d\langle M,M\rangle(s)\\
    &\quad + T\lambda \Big( \E\|A\ep{y}_0\|^2
    + \E\int_0^T \bnorm{A\ep{B}(s)Q_M^{1/2}(s)}^2_{\LL_2}
      \,d\langle M,M\rangle(s) \Big),
  \end{split}
  \\
  \begin{split}
    \norm{y_n-y_{n\lambda}}^2_{\H_2} &\lesssim \E\|y_0-J_\varepsilon^ny_0\|^2 
    + \E\int_0^T \bnorm{(B(s)-J^n_\varepsilon B(s))Q_M^{1/2}(s)}^2_{\LL_2}
        \,d\langle M,M\rangle(s)\\
    &\quad + T\lambda \Big( \E\|A_n J^n_\varepsilon y_0\|^2
    + \E\int_0^T \bnorm{A_n J^n_\varepsilon B(s)Q_M^{1/2}(s)}^2_{\LL_2}
        \,d\langle M,M\rangle(s) \Big),
  \end{split}
  \end{gather*}
  where we have set, for convenience of notation,
  $J^n_\varepsilon:=(I+\varepsilon A_n)^{-1}$. Choosing
  $\varepsilon=\lambda^{1/4}$, and recalling that, for any
  $n\in\enne$, $A_{n\lambda}$ is Lipschitz continuous with
  Lipschitz norm bounded above by $1/\lambda$, yields
  \begin{align*}
  &\lambda \Big( \E\|A_n J^n_\varepsilon y_0\|^2
  + \E\int_0^T \bnorm{A_n J^n_\varepsilon B(s)Q_M^{1/2}(s)}^2_{\LL_2}
      \,d\langle M,M\rangle(s) \Big)\\
  &\qquad = \lambda \Big( \E\|A_n(I+\lambda^{1/4}A_n)^{-1}y_0\|^2
  + \E\int_0^T \bnorm{A_n(I+\lambda^{1/4}A_n)^{-1}B(s)Q_M^{1/2}(s)}^2_{\LL_2}
      \,d\langle M,M\rangle(s)
    \Big)\\
  &\qquad = \lambda \Big( \E\|A_{n\lambda^{1/4}} y_0\|^2
  + \E\int_0^T \bnorm{A_{n\lambda^{1/4}} B(s)Q_M^{1/2}(s)}^2_{\LL_2}
      \,d\langle M,M\rangle(s) \Big)\\
  &\qquad \leq \sqrt{\lambda} \Big( \E\|y_0\|^2
  + \E\int_0^T \bnorm{B(s)Q_M^{1/2}(s)}^2_{\LL_2}\,d\langle M,M\rangle(s)
    \Big).
  \end{align*}
  It goes without saying that the same estimate holds if $A_n$ is
  replaced by $A$. We are thus left with
  \begin{align*}
  &\norm{y-y_\lambda}^2_{\H_2} + \norm{y_n-y_{n\lambda}}^2_{\H_2}\\
  &\hspace{3em} \lesssim
  \E\|y_0-J_{\lambda^{1/4}} y_0\|^2 + \E\|y_0-J_{\lambda^{1/4}}^n y_0\|^2\\
  &\hspace{3em}\quad + \E\int_0^T 
     \bnorm{(B(s)-J_{\lambda^{1/4}}B(s))Q_M^{1/2}(s)}^2_{\LL_2}
     \,d\langle M,M\rangle(s)\\
  &\hspace{3em}\quad + \E\int_0^T
     \bnorm{(B(s)-J^n_{\lambda^{1/4}}B(s))Q_M^{1/2}(s)}^2_{\LL_2}
     \,d\langle M,M\rangle(s)\\
  &\hspace{3em}\quad + T \sqrt{\lambda} \Big( \E\|y_0\|^2
    + \E\int_0^T \bnorm{B(s)Q_M^{1/2}(s)}^2_{\LL_2}
        \,d\langle M,M\rangle(s) \Big)\\
  &\hspace{3em} =: I_1 + I_2 + I_3 + I_4 + I_5.
  \end{align*}

  \medskip

  Let us now consider the second term on the right-hand side of
  \eqref{triang2}: it is immediately seen that
  $y_\lambda-y_{n\lambda}$ is the mild solution to the deterministic
  evolution equation
  \[
  (y_\lambda-y_{n\lambda})' + A_\lambda y_\lambda - A_{n\lambda}y_{n\lambda}=0,
  \qquad y_\lambda(0)-y_{n\lambda}(0)=0.
  \]
  Since $A_\lambda$ and $A_{n\lambda}$ are bounded operators, it
  follows that $y_\lambda$ and $y_{n\lambda}$ are actually strong
  solutions of \eqref{approx13} and \eqref{approx23},
  respectively. Taking scalar product of both sides with
  $y_\lambda-y_{n\lambda}$ and integrating (or, equivalently, applying
  It\^o's formula for the square of the $H$-norm), we obtain
  \[
  \frac12 \norm{y_\lambda(t)-y_{n\lambda}(t)}^2 
  + \int_0^t \ip{A_\lambda y_\lambda(s) - A_{n\lambda}y_{n\lambda}(s)}%
                {y_\lambda(s)-y_{n\lambda}(s)}\,ds = 0.
  \]
  The monotonicity of $A_{n\lambda}$ implies
  \begin{align*}
    &\ip{A_\lambda y_\lambda(s)-A_{n\lambda} y_{n\lambda}(s)}%
        {y_\lambda(s)-y_{n\lambda}(s)}\\
    &\qquad  = \ip{A_\lambda y_\lambda(s) - A_{n\lambda} y_\lambda(s)}%
                  {y_\lambda(s) - y_{n\lambda}(s)}
    + \ip{A_{n\lambda} y_\lambda(s) - A_{n\lambda} y_{n\lambda}(s)}%
         {y_\lambda(s)-y_{n\lambda}(s)}\\
    &\qquad \geq \ip{A_\lambda y_\lambda(s) - A_{n\lambda} y_\lambda(s)}%
                   {y_\lambda(s) - y_{n\lambda}(s)}
  \end{align*}
  for all $0 < s \leq T$, hence also
  \begin{align*}
    \frac12 \norm{y_\lambda(t)-y_{n\lambda}(t)}^2 &\leq
    - \int_0^t \ip{A_\lambda y_\lambda(s) - A_{n\lambda} y_\lambda(s)}%
                  {y_\lambda(s) - y_{n\lambda}(s)}\,ds\\
    &\leq \int_0^t \norm{A_\lambda y_\lambda(s) - A_{n\lambda}
    y_\lambda(s)} \norm{y_\lambda(s)-y_{n\lambda}(s)}\,ds\\
    &\leq \frac12 \int_0^t \norm{A_\lambda y_\lambda(s) - A_{n\lambda}
    y_\lambda(s)}^2\,ds 
    + \frac12 \int_0^t \norm{y_\lambda(s)-y_{n\lambda}(s)}^2\,ds,
  \end{align*}
  which in turn yields, by Gronwall's inequality and obvious estimates,
  \[
  \norm{y_\lambda-y_{n\lambda}}^2_{\H^2} \equiv
  \E\sup_{t\leq T} \|y_\lambda(t)-y_{n\lambda}(t)\|^2
  \lesssim_T \E\int_0^T \|A_\lambda y_\lambda(s)
     -A_{n\lambda} y_\lambda(s)\|^2\,ds =: I_6.
  \]
  Collecting estimates, we have
  \[
  \norm{y-y_n}^2_{\H_2} \lesssim_T \sum_{k=1}^6 I_k,
  \]
  where each $I_k$, $k=1,\ldots,6$, depends on $\lambda$ and $n$. We
  are now going to show that $\lim_{n \to \infty} I_k = 0$ for all
  $k=1,\ldots,6$. Let $\delta$ be any positive real number. Since
  $J_\lambda$ is a contraction and $J_\lambda x \to x$ as $\lambda \to
  0$ for all $x \in H$, the dominated convergence theorem implies that
  there exists $\lambda_1>0$ such that
  \[
  I_1 \equiv \E\|y_0-J_{\lambda^{1/4}} y_0\|^2 < \frac{\delta}{9} \qquad
  \forall \lambda < \lambda_1.
  \]
  By exactly the same token, there exists $\lambda_2>0$ such that
  \[
  I_3 \equiv \E\int_0^T \bnorm{(B(s)-J_{\lambda^{1/4}} B(s))Q_M^{1/2}(s)}^2_{\LL_2}
  \,d\langle M,M\rangle(s) < \frac{\delta}{9}
  \qquad \forall \lambda < \lambda_2.
  \]
  One also clearly has that there exists $\lambda_3>0$ such that
  \[
  I_5 \equiv T \sqrt{\lambda} \Big( \E\|y_0\|^2 + \E\int_0^T
  \bnorm{B(s)Q_M^{1/2}(s)}^2_{\LL_2}\,d\langle M,M\rangle(s) \Big)
  < \frac{\delta}{9} \qquad \forall \lambda < \lambda_3.
  \]
  We can safely assert that $I_1 + I_3 + I_5 < \delta/3$ for
  $\lambda=\min(\lambda_1,\lambda_2,\lambda_3)/2$. Let $\lambda$ be
  fixed from now on.

  Note that one has, by the triangle inequality and the above
  estimates,
  \begin{align*}
    I_2 + I_4 &\equiv \E\|y_0-J_{\lambda^{1/4}}^n y_0\|^2
    + \E\int_0^T \bnorm{(B(s)-J^n_{\lambda^{1/4}} B(s))Q_M^{1/2}(s)}^2_{\LL_2}
      \,d\langle M,M\rangle(s)\\
    &\leq 2\E\|y_0-J_{\lambda^{1/4}} y_0\|^2
    + 2\E\|J_{\lambda^{1/4}} y_0 - J_{\lambda^{1/4}}^n y_0\|^2\\
    &\quad + 2\E\int_0^T \bnorm{(B(s)-J_{\lambda^{1/4}}B(s))Q_M^{1/2}(s)}^2_{\LL_2}
    \,d\langle M,M\rangle(s)\\
    &\quad + 2\E\int_0^T \bnorm{(J_{\lambda^{1/4}} B(s) -
    J^n_{\lambda^{1/4}} B(s))Q_M^{1/2}(s)}^2_{\LL_2} \,d\langle M,M\rangle(s)\\
    &\leq \frac49\delta
     + 2\E\|J_{\lambda^{1/4}} y_0 - J_{\lambda^{1/4}}^n y_0\|^2\\
    &\quad + 2\E\int_0^T \bnorm{\bigl(J_{\lambda^{1/4}} B(s) -
    J^n_{\lambda^{1/4}} B(s)\bigr)Q_M^{1/2}(s)}^2_{\LL_2} \,d\langle M,M\rangle(s).
  \end{align*}
  Since $A_n \to A$ in the strong resolvent topology, by the dominated
  convergence theorem we infer that there exists $n_1>0$ such that the
  sum of the last two terms on the right-hand side of the previous
  inequality is not larger than $\delta/9$ for all $n>n_1$, i.e. that
  $\sum_{k=1}^5 I_k < 8\delta/9$ for all $n>n_1$.

  In order to conclude the proof, we only have to show that
  \[
  I_6 \equiv \E\int_0^T \|A_\lambda y_\lambda(s) -A_{n\lambda}
  y_\lambda(s)\|^2\,ds
  \]
  can be bounded by $\delta/9$ for $n$ sufficiently large. To this
  purpose, note that $A_{n\lambda} x \to A_\lambda x$ as $n \to
  \infty$ for all $x\in H$, because
  $A_{n\lambda}=\lambda^{-1}(I-\lambda J_\lambda^n)$. Therefore it is
  enough to show that the dominated convergence theorem can be
  applied. Recalling that both $A_\lambda$ and $A_{n\lambda}$
  have Lipschitz constant not larger than $1/\lambda$, one has
  \[
  \|A_\lambda y_\lambda(s) -A_{n\lambda} y_\lambda(s)\| \leq
  \|A_\lambda y_\lambda(s)\| + \|A_{n\lambda} y_\lambda(s)\| \leq
  \frac{2}{\lambda} \|y_\lambda(s)\|
  \]
  for all $s \in [0,T]$, and $y_\lambda \in L^2(\Omega \times [0,T])
  \subset \H_2$. There exists then $n_2>n_1$ such that for all $n>n_2$
  one has $I_6 < \delta/9$, hence also, by the above, $\sum_1^6
  I_k<\delta$ for all $n>n_2$, which is equivalent to
  $\lim_{n\to\infty} \norm{y-y_n}_{\H_2}=0$, thus concluding the proof.
\end{proof}

We now turn to the case $p>2$, thus providing an extension of Theorem
\ref{thm:yopsc}.

\begin{thm}     \label{thm:nyotta}
  Assume that $A_n \to A$ in the strong resolvent sense as $n \to
  \infty$ and that there exists $p>2$ such that the hypotheses on $B$
  of Proposition \ref{prop:yayo}(ii) are met. Then one has $y_n \to y$
  in $\H_p$ as $n \to \infty$, i.e.
  \[
  \lim_{n \to \infty} \E\sup_{t\leq T}\|y_n(t)-y(t)\|^p = 0.
  \]
\end{thm}
\begin{proof}
  We follow the reasoning used in the proof of the previous theorem.
  Denoting the mild solutions to \eqref{approx13} and \eqref{approx23}
  by $y_\lambda$ and $y_{n\lambda}$, respectively, one has
  \[
    \norm{y-y_n}_{\H_p} \leq \norm{y-y_\lambda}_{\H_p}
    + \norm{y_\lambda-y_{n\lambda}}_{\H_p}
    + \norm{y_{n\lambda}-y_n}_{\H_p}.
  \]
  By Corollary \ref{cor:utile} we infer that, for any
  $\varepsilon>0$ and $\lambda>0$, the following inequalities hold
  true:
  \begin{gather*}
  \begin{split}
    \norm{y-y_\lambda}^p_{\H_p} &\lesssim \E\|y_0-\ep{y}_0\|^p 
    + \E\Bigl(\int_0^T \|B(s)-\ep{B}(s)\|^2_\LL\,d[M,M](s)\Bigr)^{p/2}\\
    &\quad + (T\lambda)^{p/2} \biggl( \E\|A\ep{y}_0\|^p
    + \E\Bigl(\int_0^T \|A\ep{B}(s)\|^2_\LL\,d[M,M](s)\Bigr)^{p/2} \biggr).
  \end{split}
  \\
  \begin{split}
    \norm{y_n-y_{n\lambda}}^p_{\H_p} &\lesssim \E\|y_0-J_\varepsilon^ny_0\|^p 
    + \E\Bigl( \int_0^T
      \|B(s)-J^n_\varepsilon B(s)\|^2_\LL\,d[M,M](s)\Bigr)^{p/2}\\
    &\quad + (T\lambda)^{p/2} \biggl( \E\|A_n J^n_\varepsilon y_0\|^p
    + \E\Bigl(\int_0^T
      \|A_n J^n_\varepsilon B(s)\|^2_\LL\,d[M,M](s)\Bigr)^{p/2} \biggr).
  \end{split}
  \end{gather*}
  Choosing $\varepsilon=\lambda^{1/4}$, and recalling that, for any
  $n\in\enne$, the Lipschitz constant of $A_{n\lambda}$ is
  bounded above by $1/\lambda$, we obtain
  \begin{align*}
  &\lambda^{p/2} \biggl( \E\|A_n J^n_\varepsilon y_0\|^p
  + \E\Bigl(\int_0^T \|A_n J^n_\varepsilon B(s)\|^2_\LL\,d[M,M](s)\Bigr)^{p/2}\biggr)\\
  &\qquad = \lambda^{p/2} \biggl( \E\|A_{n\lambda^{1/4}} y_0\|^p
  + \E\Bigl( \int_0^T \|A_{n\lambda^{1/4}} B(s)\|^2_\LL\,d[M,M](s) \Bigr)^{p/2}
    \biggr)\\
  &\qquad \leq \lambda^{p/4} \biggl( \E\|y_0\|^p
  + \E\Bigl( \int_0^T \|B(s)\|^2_\LL\,d[M,M](s) \Bigr)^{p/2}
    \biggr).
  \end{align*}
  The same estimate holds if $A_n$ is replaced by $A$, therefore we
  have
  \begin{align*}
  &\norm{y-y_\lambda}^p_{\H_p} + \norm{y_n-y_{n\lambda}}^p_{\H_p}\\
  &\qquad \lesssim_p
  \E\|y_0-J_{\lambda^{1/4}} y_0\|^p + \E\|y_0-J_{\lambda^{1/4}}^n y_0\|^p\\
  &\qquad\quad + \E\Bigl(\int_0^T \|B(s)-J_{\lambda^{1/4}} B(s)\|^2_\LL
                   \,d[M,M](s)\Bigr)^{p/2}\\
  &\qquad\quad + \E\Bigl( \int_0^T \|B(s)-J^n_{\lambda^{1/4}} B(s)\|^2_\LL
                   \,d[M,M](s)\Bigr)^{p/2}\\
  &\qquad\quad + T^{p/2} \lambda^{p/4} \biggl( \E\|y_0\|^p
    + \E\Bigl( \int_0^T \|B(s)\|^2_\LL\,d[M,M](s) \Bigr)^{p/2} \Biggr)\\
  &\qquad =: I_1 + I_2 + I_3 + I_4 + I_5.
  \end{align*}
  Moreover, as in the proof of the previous theorem, we have
  \[
  \norm{y_\lambda(t)-y_{n\lambda}(t)}^2 \leq
  \int_0^t \norm{A_\lambda y_\lambda(s) - A_{n\lambda}
  y_\lambda(s)}^2\,ds + \int_0^t \norm{y_\lambda(s)-y_{n\lambda}(s)}^2\,ds,
  \]
  which in turn yields, by taking $p/2$-th power and applying
  Gronwall's inequality,
  \[
  \norm{y_\lambda-y_{n\lambda}}^p_{\H^p}
  \lesssim_{T,p} \E\int_0^T \|A_\lambda y_\lambda(s)
     -A_{n\lambda} y_\lambda(s)\|^p\,ds =: I_6,
  \]
  hence also, collecting estimates, $\norm{y-y_n}^p_{\H_p}
  \lesssim_{T,p} \sum_{k=1}^6 I_k$. Let $\delta$ be an arbitrary but
  fixed positive real number. By a reasoning already used above, we
  infer that there exist $\lambda_1$, $\lambda_2$, $\lambda_3>0$ such that
  \begin{align*}
    I_1 &\equiv \E\|y_0-J_{\lambda^{1/4}}y_0\|^p < \delta
      \quad &\forall \lambda<\lambda_1,\\
    I_3 &\equiv \E\Bigl( \int_0^T \|B(s)-J_{\lambda^{1/4}} B(s)\|^2_\LL
    \,d[M,M(s)\Bigr)^{p/2} < \delta
      \quad &\forall \lambda < \lambda_2,\\
    I_5 &\equiv T^{p/2} \lambda^{p/4} \biggl( \E\|y_0\|^p 
      + \E\Bigl( \int_0^T \|B(s)\|^2_\LL\,d[M,M](s)\Bigr)^{p/2}
      \biggr) < \delta \quad &\forall \lambda < \lambda_3,
  \end{align*}
  hence $I_1+I_3+I_5<3\delta$ for
  $\lambda:=\min(\lambda_1,\lambda_2,\lambda_3)/2$, which will remain
  fixed for the rest of the proof.
  Moreover, one has
  \begin{align*}
  I_2^{1/p} &\equiv \big\| y_0 - J^n_{\lambda^{1/4}}y_0 \big\|_{\L_p}
  \leq \big\| y_0 - J_{\lambda^{1/4}}y_0 \big\|_{\L_p}
  + \big\| J_{\lambda^{1/4}}y_0 - J^n_{\lambda^{1/4}}y_0 \big\|_{\L_p}\\
  &\leq \delta^{1/p}
  + \big\| J_{\lambda^{1/4}}y_0 - J^n_{\lambda^{1/4}}y_0 \big\|_{\L_p},\\
  I_4^{1/p} &\equiv \norm{\| B-J^n_{\lambda^{1/4}}B\|^2_\LL
    \cdot [M,M]}_{\L_p}\\
  &\leq 2 \norm{\| B-J_{\lambda^{1/4}}B\|^2_\LL \cdot [M,M]}_{\L_p} +
  2 \norm{\| J_{\lambda^{1/4}}B-J^n_{\lambda^{1/4}}B\|^2_\LL \cdot
    [M,M]}_{\L_p}\\
  &\leq 2\delta^{1/p} + 2 \norm{\|
    J_{\lambda^{1/4}}B-J^n_{\lambda^{1/4}}B\|^2_\LL \cdot
    [M,M]}_{\L_p}.
  \end{align*}
  Since $J^n_{\lambda^{1/4}} \to J_{\lambda^{1/4}}$ as $n \to
  \infty$, there exists $n_0$ such that the sum of the last two terms
  on the right-hand side of the previous inequalities is not larger
  than $\delta^{1/p}$ for all $n>n_0$, hence $\sum_{k=1}^5 I_k
  \lesssim \delta$. The proof is concluded is we show that $I_6
  \lesssim \delta$ for $n$ large enough. But this follows by observing
  that
  \[
  I_6 \equiv \E\int_0^T \|A_\lambda y_\lambda(s) -A_{n\lambda}
  y_\lambda(s)\|^p\,ds
  \]
  converges to zero as $n \to \infty$ by the dominated convergence
  theorem, $A_{n\lambda} x \to A_\lambda x$ as $n \to \infty$ for all
  $x\in H$, $y_\lambda \in L_p(\Omega \times [0,T]) \subset \H_p$, and
  $\|A_\lambda y_\lambda(s) -A_{n\lambda} y_\lambda(s)\| \leq
  2\lambda^{-1} \|y_\lambda(s)\|$.
\end{proof}

\section{Convergence of stochastic convolutions II}
\label{sec:conv2}
Let us consider the equation with Poisson random noise
\begin{equation}     \label{eq:pippa}
  dy(t)+Ay(t)\,dt = \int_Z G(t,z)\,\bar{\mu}(dt,dz), \qquad
  y(0)=y_0,
\end{equation}
where $\bar{\mu}$ is a compensated Poisson random measure, as defined in
Section \ref{sec:main}.

Consider the equations
\begin{equation}
\label{eq:yolp}
dy_\lambda(t) + A_\lambda y_\lambda(t)\,dt = 
\int_Z G(t,z)\,\bar{\mu}(dt,dz), \qquad
y(0)=y_0,
\end{equation}
and
\begin{equation}     \label{eq:rospa}
  dy_n(t) + A_ny_n(t)\,dt = \int_Z G(t,z)\,\bar{\mu}(dt,dz), \qquad
  y(0)=y_0,  
\end{equation}
where $A_\lambda$ and $A_n$ are defined as in the previous sections.

Recall that the mild solutions to \eqref{eq:pippa}, \eqref{eq:yolp}
and \eqref{eq:rospa} defined by the formula of variations of contants
are well-defined processes belonging to $\H_p$, $p \geq 2$, as soon as
$G \in L_p(\Omega\times[0,T],L_2(Z) \cap L_p(Z))$. To render notation
less burdensome, we shall denote the latter space by $\mathsf{G}_p$.

\begin{thm}
  Let $p \geq 2$ and $G \in \mathsf{G}_p$. Denoting the mild solutions
  to \eqref{eq:pippa} and \eqref{eq:yolp} by $y$ and $y_\lambda$,
  respectively, one has $y_\lambda \to y$ in $\H_p$ as $\lambda \to
  0$.
\end{thm}
\begin{proof}
  The proof is similar to the one of Theorem \ref{thm:yopsc}, and
  therefore we omit some details. One has
  \begin{equation} 
    \label{eq:trippa}
    \| y-y_\lambda \|_{\H_p} \leq \|
    y-\ep{y} \|_{\H_p} + \| \ep{y} -
    \ep{y}_\lambda \|_{\H_p} + \| \ep{y}_\lambda -
    y_\lambda \|_{\H_p},
  \end{equation}
  where $y_\lambda$ and $\ep{y}_\lambda$ are solutions to the
  regularized equations
  \begin{gather*}
    d\ep{y}(t) + A\ep{y}(t)\,dt = \int_Z
    \ep{G}(t,z)\,\bar{\mu}(dt,dz), \qquad
    y(0)=\ep{y}_0,\\
    d\ep{y}_\lambda(t) + A_\lambda \ep{y}_\lambda(t)\,dt = \int_Z
    \ep{G}(t,z)\,\bar{\mu}(dt,dz), \qquad y(0)=\ep{y}_0,
  \end{gather*}
  where $G^\varepsilon:=(I+\varepsilon A)^{-1}G$.
  Let $\delta>0$ be arbitrary but fixed. By virtue of
  \[
  \norm{y-\ep{y}}_{\H_p} + \norm{\ep{y}_\lambda - y_\lambda}_{\H_p}
  \leq \norm{y_0-\ep{y}_0}_{\L_p} + \norm{G-\ep{G}}_{\mathsf{G}_p},
  \]
  there exists $\varepsilon>0$ (which will remain fixed for the rest
  of the proof) such that $\| y-\ep{y} \|_{\H_p} + \| \ep{y}_\lambda -
  y_\lambda\|_{\H_p} < \delta/2$ for all $\lambda>0$. Recalling the
  maximal inequality (\ref{eq:**}'), the same argument used in the
  proof of Theorem \ref{thm:yopsc} yields
  \[
  \norm{y_\lambda-y_\mu}_{\H_p} \lesssim_{T,\varepsilon,p}
  \sqrt{\lambda+\mu} \bigl( \|y_0\|_{\L_p} + \|G\|_{\mathsf{G}_p} \bigr),
  \]
  hence that $\lambda \mapsto \ep{y}_\lambda$ is a Cauchy net in
  $\mathbb{H}_p$, and that
  \[
  \| \ep{y}_\lambda - \ep{y} \|_{\H_p} \lesssim_{T,\varepsilon,p}
  \sqrt{\lambda},
  \]
  which can be made smaller than $\delta/2$ choosing $\lambda$ small
  enough.
\end{proof}

In the final result of this section, we prove the analogon of Theorem
\ref{thm:nyotta} for equations driven by Poisson random noise.
\begin{thm}     \label{thm:trippona}
  Assume that $A_n \to A$ in the strong resolvent sense as $n \to
  \infty$, and that there exists $p \in [2,\infty$ such that $G \in
  \mathsf{G}_p$. Denoting the mild solutions to \eqref{eq:pippa} and
  \eqref{eq:rospa} by $y$ and $y_n$, respectively, one has $y_n
  \to y$ in $\H_p$ as $n \to \infty$.
\end{thm}
\begin{proof}
  The proof follows the same lines of the proofs of Theorems
  \ref{thm:nyo2sc} and \ref{thm:nyotta}, therefore we omit some
  details. Let $y_\lambda$ and $y_{n\lambda}$ be the solutions to the
  regularized equations
  \begin{gather*}
    dy_\lambda(t)+A_\lambda y_\lambda(t)\,dt= \int_Z G(t,z)\,\m(dt,dz),
    \qquad y_\lambda(0)=y_0,\\
    dy_{n\lambda}(t)+A_{n\lambda} y_{n\lambda}(t)\,dt = \int_Z G(t,z)\,\m(dt,dz),
    \qquad y_{n\lambda}(0)=y_0,
  \end{gather*}
  and observe that
  \[
    \norm{y-y_n}_{\H_p} \leq \norm{y-y_\lambda}_{\H_p}
    + \norm{y_\lambda-y_{n\lambda}}_{\H_p}
    + \norm{y_{n\lambda}-y_n}_{\H_p}.
  \]
  In analogy to Corollary \ref{cor:utile}, one has, for all
  $\lambda>0$, $\varepsilon>0$,
  \[
    \norm{y-y_\lambda}_{\H_p} + \norm{y_{n\lambda}-y_n}_{\H_p}
    \lesssim \sum_{k=1}^5 I_k,
  \]
  where
  \begin{align*}
    I_1 &:= \norm{y_0-J_{\lambda^{1/4}}y_0}_{\L_p}
    \xrightarrow{\lambda \to 0} 0,\\
    I_2 &:= \| y_0-J^n_{\lambda^{1/4}}y_0 \|_{\L_p}
    \leq \| y_0-J_{\lambda^{1/4}}y_0 \|_{\L_p} 
    + \| J_{\lambda^{1/4}}y_0-J^n_{\lambda^{1/4}}y_0 \|_{\L_p}\\
    I_3 &:= \| G-J_{\lambda^{1/4}}G \|_{\mathsf{G}_p}
    \xrightarrow{\lambda \to 0} 0,\\
    I_4 &:= \| G-J^n_{\lambda^{1/4}}G \|_{\mathsf{G}_p}
    \leq \| G-J_{\lambda^{1/4}}G \|_{\mathsf{G}_p}
    + \| J_{\lambda^{1/4}}G-J^n_{\lambda^{1/4}}G \|_{\mathsf{G}_p}\\
    I_5 &:= T^{1/2} \lambda^{1/4} \bigl( \|y_0 \|_{\L_p} +
    \|G\|_{\mathsf{G}_p} \bigr)\xrightarrow{\lambda \to 0} 0,
  \end{align*}
  and
  \[
  \| J_{\lambda^{1/4}}y_0-J^n_{\lambda^{1/4}}y_0 \|_{\L_p}
  + \| J_{\lambda^{1/4}}G-J^n_{\lambda^{1/4}}G \|_{\mathsf{G}_p}
  \xrightarrow{n \to \infty} 0 \qquad \forall \lambda>0,
  \]
  that is, for any given $\delta>0$, one can choose $\lambda$ and
  $n_0>0$ such that $\norm{y-y_\lambda}_{\H_p} +
  \norm{y_{n\lambda}-y_n}_{\H_p}<\delta$ for all $n>n_0$.  The proof
  is finished noting that
  \[
  \norm{y_\lambda-y_{n\lambda}}^p_{\H_p} \lesssim \E\int_0^T
  \|A_\lambda y_\lambda(s) -A_{n\lambda} y_\lambda(s)\|^p\,ds,
  \]
  which converges to zero as $n \to \infty$ by the dominated
  convergence theorem and $y_\lambda \in \H_p$.
\end{proof}

\section{Proof of the main results}     \label{sec:prove}
By definition of mild solution, one has, in the case of Theorem
\ref{thm:nyo2},
\begin{align*}
  u(t) &= S(t)u_0 - \int_0^t S(t-s)f(u(s))\,ds + \int_0^t
  S(t-s)B(u(s-))\,dM(s),\\
  u_n(t) &= S_n(t)u_{0n} - \int_0^t S_n(t-s)f_n(u_n(s))\,ds + \int_0^t
  S_n(t-s)B_n(u_n(s-))\,dM(s),
\end{align*}
and, in the case of Theorem \ref{thm:nyop},
\begin{align*}
  u(t) &= S(t)u_0 - \int_0^t S(t-s)f(u(s))\,ds + \int_0^t\!\int_Z
  S(t-s)G(z,u(s-))\,\m(ds,dz),\\
  u_n(t) &= S_n(t)u_{0n} - \int_0^t S_n(t-s)f_n(u_n(s))\,ds
  + \int_0^t\!\int_Z S_n(t-s)G_n(z,u_n(s-))\,\m(ds,dz).
\end{align*}
We shall set, for notational convenience,
\begin{gather*}
S \ast \phi(t) := \int_0^t S(t-s)\phi(s)\,ds,\\
S \diamond \Phi(t) := \int_0^t S(t-s)\Phi(s)\,dM(s),
\qquad S \diamond \Psi(t) := \int_0^t\!\int_Z S(t-s)\Psi(z,s)\,\m(ds,dz)
\end{gather*}
(even though we use the same symbol to denote stochastic convolutions
with respect to a martingale and to a compensated Poisson measure,
there will be no risk of confusion).

In the case of Theorem \ref{thm:nyo2}, the triangle inequality yields,
for any $0 < t \leq T$,
\begin{equation}     \label{eq:pilu}
\begin{aligned}
\|u - u_n \|_{\H_p(t)} &\leq
\|Su_0 - S_nu_{0n}\|_{\H_p(t)}
+ \| S \ast f(u) - S_n \ast f_n(u_n) \|_{\H_p(t)}\\
&\quad + \| S \diamond B(u) - S_n \diamond B_n(u_n) \|_{\H_p(t)}.
\end{aligned}
\end{equation}
The same estimate holds in the case of Theorem \ref{thm:nyop},
replacing $B$ and $B_n$ with $G$ and $G_n$, respectively.

\begin{lemma}     \label{lm:uno}
  Let $p \geq 2$ and $0\leq t \leq T$. If $\xi_n \to \xi$ in $\L_p$
  as $n \to \infty$, then $S_n\xi_n \to S\xi$ in $\H_p(t)$.
\end{lemma}
\begin{proof}
  It clearly suffices to prove the claim for $t=T$.  By the triangle
  inequality we can write
  \[
  \|S_n\xi_n - S\xi\|_{\H_p} \leq 
  \|S_n\xi_n - S_n\xi\|_{\H_p}
  + \|S_n\xi - S\xi\|_{\H_p}.
  \]
  Since $\sup_{t \leq T} \|S_n(t)\xi - S(t)\xi\|^p \leq e^{p\eta T}
  \|\xi_n-\xi\|^p$, taking expectations on both sides and passing to
  the limit, we get $\|S_n(t)\xi_n - S_n(t)\xi\|_{\H_p} \to 0$ as $n
  \to \infty$. Moreover, by the Trotter-Kato's theorem,
  $S_n(\cdot)\xi$ converges to $S(\cdot)\xi$ $\P$-a.s. uniformly on
  compact sets, i.e.
  \[
  \lim_{n \to \infty} \sup_{t \leq T} \| S_n(t)\xi - S(t)\xi \| = 0,
  \]
  which implies, together with
  \[
  \sup_{t \leq T} \|S_n(t)\xi - S(t)\xi \|^p \lesssim e^{p \eta T}
  \|\xi\|^p
  \]
  and $\E\|\xi\|^p<\infty$, that $\|S_n(t)\xi - S(t)\xi\|_{\H_p} \to
  0$ as $n \to \infty$, thanks to the dominated convergence theorem.
\end{proof}

\begin{lemma}     \label{lm:due}
  Let $p \geq 2$, $0 < t \leq T$, and $v$, $w \in \H_p$. For every
  $\delta>0$ there exist $n_0 \in \enne$ and $\gamma>0$, independent
  of $\delta$ and $n_0$, such that
  \[
  \| S_n \ast f_n(v) - S \ast f(w) \|^p_{\H_p(t)} \leq
  \delta + \gamma \int_0^t \| v - w \|^p_{\H_p(s)}\,ds
  \]
  for all $n>n_0$.
\end{lemma}
\begin{proof}
  The triangle inequality yields
  \begin{align*}
    \| S_n \ast f_n(v) - S \ast f(w) \|^p_{\H_p(t)} &\leq
    3^p \| S_n \ast f_n(v) - S_n \ast f_n(w) \|_{\H_p(t)}\\
    &\quad + 3^p \| S_n \ast f_n(w) - S_n \ast f(w) \|_{\H_p(t)}\\
    &\quad + 3^p \| S_n \ast f(w) - S \ast f(w) \|_{\H_p(t)}.
  \end{align*}
  Recalling that the operator norm of $S_n(t)$ is bounded by $e^{\eta
    t}$, by H\"older's inequality one has that there exists a constant
  $N=N(T,p,\eta)$ such that
  \begin{align*}
  &\| S_n \ast f_n(v) - S_n \ast f_n(w) \|^p_{\H_p(t)}\\
  &\qquad = \E\sup_{s \leq t} \Bigl\| \int_0^s S_n(s-r)
  \bigl(f_n(v(r))-f_n(w(r))\bigr)\,dr \Bigr\|^p\\
  &\qquad \leq N \, \E\sup_{s \leq t} \int_0^s 
  \bigl\|f_n(v(s))-f_n(w(s))\bigr\|^p\,dr\\
  &\qquad \leq N L_f^p \int_0^t \E\sup_{r\leq s}\|v(r)-w(r)\|^p\,dr\\
  &\qquad = N L_f^p \int_0^t \|v-w\|^p_{\H_p(s)}\,dr,
  \end{align*}
  as well as
  \[
  \| S_n \ast f_n(w) - S_n \ast f(w) \|^p_{\H_p(t)}
  \lesssim_{T,p,\eta} \E\int_0^T \|f_n(w(s)) - f(w(s))\|^p\,ds,
  \]
  where the right-hand side converges to zero as $n \to \infty$ by the
  dominated convergence theorem: in fact, $f_n \to f$ pointwise
  implies $f_n(w(s)) \to f(w(s))$ for all $s \leq T$, and
  \[
  \|f_n(w) - f(w)\| \leq 2L_f \|w\| \in L_p(\Omega\times[0,T]).
  \]
  Therefore there exists $n_1$ such that $\|S_n \ast f_n(w) - S_n \ast
  f(w) \|^p_{\H_p(t)}<3^{-p}\delta/2$ for all $n>n_1$. Theorem
  \ref{thm:titikaka} implies
  \[
  \lim_{n \to \infty} \sup_{t\leq T}\, \bnorm{\bigl(S_n \ast f(w))(t)
    - \bigl(S \ast f(w))(t)} = 0,
  \]
  which implies
  \[
  \| S_n \ast f(w) - S \ast f(w) \|^p_{\H_p(T)} \to 0
  \]
  as $n \to \infty$ by the dominated convergence theorem. In fact, one
  has
  \begin{align*}
  &\E\sup_{t\leq T}\, \bnorm{\bigl(S_n \ast f(w))(t)
      - \bigl(S \ast f(w))(t)}^p\\
  &\hspace{3em} \lesssim
  \E\sup_{t\leq T}\, \bnorm{S_n \ast f(w))(t)}^p
  + \E\sup_{t \leq T} \bnorm{S \ast f(w))(t)}^p\\
  &\hspace{3em} \lesssim_{\eta,\|f\|_{\lip}} \|w\|^p_{\H_p} < \infty.
  \end{align*}
  Therefore there exists $n_2 \in \enne$ such that $\| S_n \ast
  f(w) - S \ast f(w) \|^p_{\H_p(t)} < 3^{-p}\delta/2$ for all $n>n_2$. The
  proof is completed taking $n_0=\max(n_1,n_2)$ and $\gamma=3^pNL_f^p$.
\end{proof}

\subsection{Proof of Theorem \ref{thm:nyo2}}
\label{ssec:prova1}
The following estimate is crucial for the proof of the theorem.
\begin{lemma}     \label{lm:tre}
  Let $0 < t \leq T$. For every $\delta>0$ there exist $n_0 \in
  \enne$ and $\gamma>0$, independent of $\delta$ and $n_0$, such
  that
  \[
  \| S_n \diamond B_n(u_{n-}) - S \diamond B(u_-) \|^2_{\H_2(t)} \leq
  \delta + \gamma \int_0^t \| u_n - u \|^2_{\H_2(s)}\,ds
  \]
  for all $n>n_0$.
\end{lemma}
\begin{proof}
  The triangle inequality yields
  \begin{align*}
    \norm{S_n \diamond B_n(u_{n-}) - S \diamond B(u_-)}^2_{\H_2(t)} &\leq
    9 \norm{S_n \diamond B_n(u_{n-}) - S_n \diamond B_n(u_-)}^2_{\H_2(t)}\\
    &\quad + 9 \norm{S_n \diamond B_n(u_-) - S_n \diamond B(u_-)}^2_{\H_2(t)}\\
    &\quad + 9 \norm{S_n B(u_-) - S \diamond B(u_-)}^2_{\H_2(t)}.
  \end{align*}
  Thanks to the maximal inequality \eqref{eq:maxi2}, there exists a
  constant $N=N(\eta)$ such that
  \begin{align*}
  &\norm{S_n \diamond B_n(u_{n-})  - S_n \diamond B_n(u_-) }^2_{\H_2(t)}\\
  &\qquad = \E\sup_{s \leq t} \Bigl\| \int_0^s S_n(s-r)
    \bigl(B_n(u_n(r-))-B_n(u(r-))\bigr)\,dM(r) \Bigr\|^2\\
  &\qquad \leq N\,\E\int_0^t \bigl\|
    \bigl(B_n(u_n(s-))-B_n(u(s-))\bigr)Q_M^{1/2}(s)\bigr\|^2_{\LL_2}
    \,d\langle M,M\rangle(s)\\
  &\qquad \leq N L_B^2\E\int_0^t \sup_{r\leq s}\norm{u_n(r)-u(r)}^2\,ds\\
  &\qquad = N L_B^2 \int_0^t \norm{u_n-u}^2_{\H_2(s)}\,ds,
  \end{align*}
  as well as
  \begin{align*}
  &\norm{S_n \diamond B_n(u_-)  - S_n \diamond B(u_-) }^2_{\H_2(t)}\\
  &\qquad \lesssim_\eta \E\int_0^t \bigl\|
    \bigl(B_n(u(s-))-B(u(s-))\bigr)Q_M^{1/2}\bigr\|^2_{\LL_2}
    \,d\langle M,M\rangle(s)\\
  &\qquad \leq \E\int_0^T \bigl\| B_n(u(s))-B(u(s)) \bigr\|^2_Q\,ds,
  \end{align*}
  which converges to zero by pointwise convergence of $B_n$ to $B$ and
  the dominated convergence theorem, taking into account that
  \[
  \bigl\| B_n(u)-B(u) \bigr\|_Q \lesssim 2\|u\| 
  \in L_2(\Omega \times [0,T]).
  \]
  Therefore there exists $n_1 \in \enne$ such that $\norm{S_n
    \diamond B_n(u_-) - S_n \diamond B(u_-) }^2_{\H_2(t)}<\delta/18$ for
  all $n>n_1$.  Finally, Theorem \ref{thm:nyo2sc} implies that there
  exists $n_2 \in \enne$ such that
  \begin{align*}
    \norm{S_n \diamond B(u_-) - S \diamond B(u_-)}^2_{\H_2(t)} <
    \frac{\delta}{18}
  \end{align*}
  for all $n>n_2$. The proof is completed setting $n_0=\max(n_1,n_2)$
  and $\gamma=9 N L_f^2$.
\end{proof}

\begin{proof}[Proof of Theorem \ref{thm:nyo2}]
  By \eqref{eq:pilu} we have, for any $0 < t \leq T$,
  \begin{align*}
    \|u - u_n \|^2_{\H_2(t)} &\leq
    9\|Su_0 - S_nu_{0n}\|^2_{\H_2(t)}
    + 9\| S \ast f(u) - S_n \ast f_n(u_n) \|^2_{\H_2(t)}\\
    &\quad + 9\| S \diamond B(u) - S_n \diamond B_n(u_n) \|^2_{\H_2(t)}.
  \end{align*}
  Let $\varepsilon>0$ be arbitrary and $\varepsilon'=e^{-\gamma
    T}\varepsilon$, where $\gamma$ is a constant whose value will be
  specified later.  By Lemma \ref{lm:uno} there exists $n_1\in\enne$,
  independent of $t$, such that
  \[
  9\|Su_0 - S_nu_{0n}\|^2_{\H_2(t)} < \frac{\varepsilon'}{3} \qquad
  \forall n>n_1.
  \]
  Similarly, by Lemmata \ref{lm:due} and \ref{lm:tre} there exist
  $n_2$, $n_3 \in \enne$, independent of $t$, and $\gamma_1$,
  $\gamma_2>0$, depending only on the Lipschitz constants of $f$ and
  $B$, such that
  \begin{align*}
    9\| S \ast f(u) - S_n \ast f_n(u_n) \|^2_{\H_2(t)}
    &< \frac{\varepsilon'}{3} + \gamma_1 \int_0^t \norm{u_n-u}^2_{\H_2(s)}\,ds
    &\forall n>n_2,\\
    9\| S \diamond B(u) - S_n \diamond B_n(u_n) \|^2_{\H_2(t)}
    &< \frac{\varepsilon'}{3} + \gamma_2 \int_0^t \norm{u_n-u}^2_{\H_2(s)}\,ds
    &\forall n>n_3.
  \end{align*}
  Therefore, setting $n_0=n_1+n_2+n_3$ and $\gamma=\gamma_1+\gamma_2$,
  we are left with
  \[
  \norm{u_n-u}^2_{\H_2(t)} < \varepsilon' + \gamma \int_0^t
  \norm{u_n-u}^2_{\H_2(s)}\,ds \qquad \forall t \in ]0,T],
  \]
  which implies, by Gronwall's inequality,
  \[
  \norm{u_n-u}^2_{\H_2(T)} < e^{\gamma T} \varepsilon' = \varepsilon
  \qquad \forall n>n_0.
  \]
  Since $\varepsilon$ is arbitrary, this is equivalent to asserting
  that
  \[
  \lim_{n\to\infty} \norm{u_n-u}_{\H_2(T)} = 0 \qedhere
  \]
\end{proof}

\subsection{Proof of Theorem \ref{thm:nyop}}
As in the previous subsection, we establish first a key estimate.
\begin{lemma}     \label{lm:treppe}
  Let $2 \leq p < \infty$ and $0 < t \leq T$. For every $\delta>0$
  there exist $n_0 \in \enne$ and $\gamma>0$, independent of $\delta$
  and $n_0$, such that
  \[
  \| S_n \diamond G_n(u_{n-}) - S \diamond G(u_-) \|^p_{\H_p(t)} \leq
  \delta + \gamma \int_0^t \| u_n - u \|^p_{\H_p(s)}\,ds
  \]
  for all $n>n_0$.
\end{lemma}
\begin{proof}
  The triangle inequality yields
  \begin{align*}
    \norm{S_n \diamond G_n(u_{n-}) - S \diamond G(u_-)}^p_{\H_p(t)} &\lesssim_p
    \norm{S_n \diamond G_n(u_{n-}) - S_n \diamond G_n(u_-)}^p_{\H_p(t)}\\
    &\quad + \norm{S_n \diamond G_n(u_-) - S_n \diamond G(u_-)}^p_{\H_p(t)}\\
    &\quad + \norm{S_n G(u_-) - S \diamond G(u_-)}^p_{\H_p(t)},
  \end{align*}
  where the implicit constant in the inequality can be taken equal to
  $N_1=N_1(p)=3^p$.  Thanks to the maximal inequality
  \eqref{eq:**}, there exists a constant $N_2=N_2(\eta,p)$ such that
  \begin{align*}
  \norm{S_n \diamond G_n(u_{n-})  - S_n \diamond G_n(u_-) }^p_{\H_p(t)}
  &\leq N_2\, \E\int_0^t \bigl\| G_n(u_n(s)) - G_n(u(s))
     \bigr\|^p_{L_p(Z)\cap L_2(Z)}\,ds\\
  &\leq N_2 L_G^p \int_0^t \norm{u_n-u}^p_{\H_p(s)}\,ds.
  \end{align*}
  Similarly, one also has
  \[
  \norm{S_n \diamond G_n(u_-)  - S_n \diamond G(u_-) }^p_{\H_p(t)}
  \lesssim_{\eta,p} \E\int_0^t \bigl\|
    G_n(u(s))-G(u(s)) \bigr\|^p_{L_p(Z)\cap L_2(Z)}\,ds
  \]
  which converges to zero by pointwise convergence of $G_n$ to $G$ and
  the dominated convergence theorem. In particular, there exists $n_1
  \in \enne$ such that
  \[
  \norm{S_n \diamond G_n(u_-) - S_n \diamond G(u_-)}^p_{\H_p(t)}
  < \frac12\,\frac{\delta}{3^p} \qquad \forall n>n_1.
  \]
  Finally, Theorem \ref{thm:trippona} implies that there exists $n_2
  \in \enne$ such that
  \begin{align*}
    \norm{S_n \diamond G(u_-) - S \diamond G(u_-)}^p_{\H_p(t)} <
    \frac12\,\frac{\delta}{3^p} \qquad \forall n>n_2.
  \end{align*}
  The proof is completed setting $n_0=\max(n_1,n_2)$ and $\gamma=N_1 N_2
  L_G^p$.
\end{proof}

\begin{proof}[Proof of Theorem \ref{thm:nyop}]
  By \eqref{eq:pilu} we have, for any $0 < t \leq T$,
  \begin{align*}
    \|u - u_n \|^p_{\H_p(t)} &\leq
    N_1 \|Su_0 - S_nu_{0n}\|^p_{\H_p(t)}
    + N_1 \| S \ast f(u) - S_n \ast f_n(u_n) \|^p_{\H_p(t)}\\
    &\quad + N_1 \| S \diamond G(u) - S_n \diamond G_n(u_n) \|^p_{\H_p(t)},
  \end{align*}
  where $N_1=3^p$. Let $\varepsilon>0$ be arbitrary and
  $\varepsilon'=e^{-\gamma T}\varepsilon$, where $\gamma$ is a
  constant whose value will be specified later.  By Lemma \ref{lm:uno}
  there exists $n_1\in\enne$, independent of $t$, such that
  \[
  N_1 \|Su_0 - S_nu_{0n}\|^p_{\H_p(t)} < \frac{\varepsilon'}{3} \qquad
  \forall n>n_1.
  \]
  Similarly, by Lemmata \ref{lm:due} and \ref{lm:treppe} there exist
  $n_2$, $n_3 \in \enne$, independent of $t$, and $\gamma_1$,
  $\gamma_2>0$, depending only on the Lipschitz constants of $f$ and
  $G$, such that
  \begin{align*}
    N_1\| S \ast f(u) - S_n \ast f_n(u_n) \|^p_{\H_p(t)}
    &< \frac{\varepsilon'}{3} + \gamma_1 \int_0^t \norm{u_n-u}^p_{\H_p(s)}\,ds
    &\forall n>n_2,\\
    N_1\| S \diamond G(u) - S_n \diamond G_n(u_n) \|^p_{\H_p(t)}
    &< \frac{\varepsilon'}{3} + \gamma_2 \int_0^t \norm{u_n-u}^p_{\H_p(s)}\,ds
    &\forall n>n_3.
  \end{align*}
  Therefore, setting $n_0=n_1+n_2+n_3$ and $\gamma=\gamma_1+\gamma_2$,
  we are left with
  \[
  \norm{u_n-u}^p_{\H_p(t)} < \varepsilon' + \gamma \int_0^t
  \norm{u_n-u}^p_{\H_p(s)}\,ds \qquad \forall t \in ]0,T],
  \]
  which implies, by Gronwall's inequality,
  \[
  \norm{u_n-u}^p_{\H_p(T)} < e^{\gamma T} \varepsilon' = \varepsilon
  \qquad \forall n>n_0,
  \]
  i.e. 
  \[
  \lim_{n\to\infty} \norm{u_n-u}_{\H_p(T)} = 0
  \]
  because $\varepsilon$ is arbitrary.
\end{proof}

\section{On equations with additive martingale noise}
It was observed in Remark \ref{rmk:tanti} that the reason for
considering equations driven by $M$ only in $\H_2$ is that we do not
know whether it is possible to find a Lipschitz condition involving
only $B$ (and at most a ``deterministic'' quantity depending on $M$,
such as $Q$) such that a fixed point argument could be used to prove
well-posedness in $\H_p$. In the case of equations with additive noise
the problem disappears, and it is immediate to deduce the following
result.
\begin{thm}
  Let $p>2$. Assume that hypotheses (i) and (ii) of Theorem
  \ref{thm:nyo2} are satisfied, and that
  \begin{itemize}
  \item[\emph{(iii')}] $B$, $B_n$ are predictable $\LL(K,H)$-valued processes
    such that
    \[
    \bigl(\|B\|^2_\LL \cdot [M,M]\bigr)^{1/2} \in \L_p,
    \qquad \bigl(\|B_n - B\|^2_\LL \cdot [M,M]\bigr)^{1/2}
    \xrightarrow{n\to\infty} 0 \quad \text{in } \L_p;
    \]
  \item[\emph{(iv')}] $u_{0n} \to u_0$ in $\L_p$ as $n\to\infty$.
  \end{itemize}
  Denoting by $u$ and $u_n$ the mild solutions to
  \[
  du + Au\,dt + f(u)\,dt = B\,dM, \qquad u(0)=u_0,
  \]
  and
  \[
  du_n + A_nu_n\,dt + f_n(u_n)\,dt = B_n\,dM, \qquad u(0)=u_{0n},
  \]
  respectively, one has $u_n \to u$ in $\H_p$.
\end{thm}

\begin{rmk}     \label{rmk:noloss}
  As mentioned at the beginning of Section \ref{sec:conv1}, Theorems
  \ref{thm:yo2sc}, \ref{thm:yopsc}, \ref{thm:nyo2sc} and
  \ref{thm:nyotta} remain true also assuming that $A+\eta I$ is maximal
  monotone. Let us consider the case of $A_n \to A$ in the strong
  resolvent sense. Then $y_n$ is the mild solution to
  \[
  dy_n + \tilde{A}_n y_n\,dt - \eta y_n\,dt = B\,dM, \qquad y_n(0)=y_0,
  \]
  where $\tilde{A}_n:=A_n+\eta I$. Setting $f=f_n:=\eta I$ for all
  $n\in\enne$, the previous theorem implies that $y_n \to y$ in
  $\H_p$, with $p$ depending on the hypotheses on $B$ and $M$.
\end{rmk}

\bibliographystyle{amsplain}
\bibliography{ref}

\end{document}